\documentclass[12pt]{article}
\usepackage{amsmath, amsthm, amssymb, color}
\usepackage[colorlinks=true,linkcolor=blue,urlcolor=blue]{hyperref}
\usepackage{graphicx}
\usepackage{caption}
\usepackage{mathtools}
\usepackage{enumerate}
\usepackage{verbatim}
\usepackage{mathrsfs}
\usepackage{booktabs}
\usepackage{tikz,tikz-cd,tikz-3dplot}
\usepackage{amssymb}
\usetikzlibrary{matrix}
\usetikzlibrary{arrows}
\usetikzlibrary{positioning}
\usepackage{algorithm}
\usepackage[noend]{algpseudocode}
\usepackage{caption}
\usepackage{titlesec}
\usepackage[normalem]{ulem}
\usepackage{subcaption}
\oddsidemargin \evensidemargin
\usepackage{makecell}
\usepackage{array}

\newtheorem{theorem}{Theorem}
\newtheorem{proposition}[theorem]{Proposition}
\newtheorem{lemma}[theorem]{Lemma}
\newtheorem{corollary}[theorem]{Corollary}

\theoremstyle{definition}
\newtheorem{definition}[theorem]{Definition}

\newtheorem{remark}[theorem]{Remark}
\newtheorem{conjecture}[theorem]{Conjecture}

\newtheorem{example}[theorem]{Example}

\numberwithin{theorem}{section}

\usepackage{thm-restate}

\newcommand{\PP}{\mathbb{P}}

\newcommand{\CC}{\mathbb{C}}

\newcommand{\NN}{\mathbb{N}}

\newcommand{\rank}{\mathrm{rank}}

\newcommand{\V}{\mathcal{V}}

\titleformat{\subsection}[runin]
  {\normalfont\normalsize\bfseries}
  {\thesubsection}
  {1em}
  {}[.]

\titlespacing*{\subsection}
  {0pt}{\baselineskip}{1em}

\title{Equations of Tree Tensor Network Varieties}
\author{Serkan Ho\c{s}ten, Niharika Chakrabarty Paul, Otto T.~P.~Schmidt, Dmitry Skurt}
\date{}

\begin{document}

\maketitle

\begin{abstract}
We show that tree tensor network varieties, including tensor train varieties, are general Markov models associated to spaced trees. This allows us to prove that the prime ideals of these varieties are generated by minors  of  matrix flattenings. In the case of tensor train varieties, we discuss whether these minors form a Gr\"obner basis and provide a combinatorial method to compute the degree for order $3$ tensor trains.
\end{abstract}

\section{Introduction}
Tree tensor networks (TTN), or hierarchical Tucker decompositions, were first introduced 
in \cite{HK09} to represent high-order tensors. The representation depends on a combinatorial tree $\mathscr{T}$ which encodes the decomposition of a tensor via a sequential application of singular value or QR decompositions. Due to the general applicability of structured tensor decomposition, TTN play a significant role in various areas of physics and computer science. They have been utilized in the simulation of time-evolution of spin systems \cite{krinitsin2025}, in machine learning for computer vision \cite{Cheng_2019} and to describe holographic duality in AdS/CFT \cite{Hayden_2016}. Especially in the area of quantum many-body physics these networks have gained popularity for enabling the description of two-dimensional systems while allowing a unique isometry center in the network \cite{humpert2025, krinitsin2025}.

Two instances of TTN stand out in practical applications: subspace (or Tucker) representations and tensor trains (or matrix product states). The former was already recognized as a TTN in \cite{HK09} and the latter was introduced in \cite{Oseledets_2011}. We will approach general TTN and these special cases with a view from algebraic geometry; see \cite{BUCZYNSKA_2015} for binary TTN, \cite{Breiding_2024}, \cite{Landsberg2012}, and \cite{Ye2018} for subspace representations, and \cite{BUCZYNSKA_2015} and \cite{BFHP25} for tensor trains from this perspective. 

Besides the tree $\mathscr{T}$, the data which determines a TTN includes a sequence $\mathbf{V} = (V_\lambda)_{\lambda \in \mathscr{L}}$ of complex vector spaces attached to the set of leaves $\mathscr{L}$ of $\mathscr{T}$ and a sequence of positive integers
$\mathbf{r} = (r_v)_{v \in \mathscr{T}\setminus \{\rho\}}$ associated to every non-root vertex of $\mathscr{T}$ (Definition \ref{def:subspace_def}). The set of all tensors in $\PP(\bigotimes_{\lambda \in \mathscr{L}} V_\lambda)$ determined by $(\mathscr{T}, \mathbf{V}, \mathbf{r})$ constitutes a projective algebraic variety $\mathrm{TTN}_{\mathscr{T}, \mathbf{V}, \mathbf{r}}$ which we call a \emph{tree tensor network variety}. In particular, $\mathrm{TTN}_{\mathscr{T}, \mathbf{V}, \mathbf{r}}$ is the set of all tensors $[\psi] \in \PP(\bigotimes_{\lambda \in \mathscr{L}} V_\lambda)$ with 
$\rank(\psi^{(v)}) \leq r_v$ for all $v \in \mathscr{T} \setminus \{\rho\}$ where $\psi^{(v)}$ is a matrix flattening of $\psi$; see Proposition \ref{prop:TTN_set_theoretic}. Hence, $\mathrm{TTN}_{\mathscr{T}, \mathbf{V}, \mathbf{r}}$ is the common zero set of minors of size $r_v+1$ of $\psi^{(v)}$. The most natural question is: \emph{do these minors generate the prime ideal of $\mathrm{TTN}_{\mathscr{T}, \mathbf{V}, \mathbf{r}}$, i.e., do they define the tree tensor network variety scheme-theoretically and not just set-theoretically?} 
We answer this question to the positive. 

\begin{restatable}{theorem}{main}\label{thm: main}
   For a fixed tree $\mathscr{T}$, the vector spaces $\mathbf{V} = (V_\lambda)_{\lambda \in \mathscr{L}}$, and a rank sequence $\mathbf{r} = (r_v)_{v \in \mathscr{T}\setminus\{\rho\}}$, let $\psi \in \bigotimes_{\lambda \in \mathscr{L}} V_\lambda$ be a  tensor of indeterminates and let $\psi^{(v)}$ for $v \in \mathscr{T}\setminus \{\rho\}$ be the flattenings as in \eqref{eq:flattening}. Then the homogeneous prime ideal of $\mathrm{TTN}_{\mathscr{T},\mathbf{V}, \mathbf{r}}$ in $\mathbb{C}[\psi]$ is
\[
\mathcal I(\mathrm{TTN}_{\mathscr{T},\mathbf{V}, \mathbf{r}})
= \sum_{v \in \mathscr{T}\setminus \{\rho\}} \Big\langle (r_{v} + 1)\textrm{-minors of }\psi^{(v)} \Big\rangle.
\]
\end{restatable}
We note that this result is known for the special case of subspace representations \cite{landsberg2007}, and in the case of tensor trains it appeared as a conjecture \cite[Conjecture 5.10]{BFHP25}. Our proof will come in two steps. For the first step, we will consider a different algebraic variety that is defined on a tree, namely, the general Markov model $\V_T$ on the spaced tree $T$ \cite{DK09}. Here, the spaced tree $T$ refers to a combinatorial tree, as well as the vector and matrix spaces attached to all of its vertices and edges, respectively. Given a TTN defined by $(\mathscr{T}, \mathbf{V}, \mathbf{r})$, we construct the corresponding spaced tree $T$ and prove the following. 

\begin{restatable}{theorem}{TTNisSpacedTree} \label{thm:TTN=Spaced_tree}
The tree tensor network variety $\mathrm{TTN}_{\mathscr{T}, \mathbf{V}, \mathbf{r}}$ is equal to the general Markov model $\V_T$ on the spaced tree $T$ constructed in the preamble of Section \ref{sec:TTN_are_spaced_trees}.
\end{restatable}

We wish to remark that \cite[Theorem 2.1]{RS19} proves a result of a similar spirit, translating tensor hypernetworks (which include subspace representations and tensor trains) to discrete graphical models associated to a hypergraph. In Theorem \ref{thm:TTN=Spaced_tree} we translate tree tensor network varieties (which also include subspace representations and tensor trains) to general Markov models on spaced trees. Our approach brings the advantage of characterizing the defining ideal of tree tensor network varieties.
This is the second step in our proof of Theorem \ref{thm: main} and is based on the description of the ideal $\mathcal{I}(\V_T)$ given by \cite[Theorem 1.7]{DK09}. This ideal description contains additional ideals besides the determinantal ideals in Theorem \ref{thm: main}. In a crucial step in Section \ref{sec:equations} we show that these additional ideals are already included in the determinantal ideals.   

Our result on the primeness of the TTN ideal gives a rigorous definition of the Zariski tangent space of TTN varieties. This provides a new, intrinsic method to verify and construct tangent spaces of TTN varieties, as used for so-called tangent-space \emph{ans\"atze} in the study of correlated many-body systems \cite{Bauernfeind_2020, Haegeman_2011} or Riemannian optimization using TTN \cite{willner2025}. More specifically, the primeness of the ideal allows us to realize the tangent space independently of a parametrization and so-called \emph{gauge redundancies} via the Jacobian of the ideal generators. That might prove particularly useful for computational methods such as the \emph{time-dependent variational principle}  \cite{Bauernfeind_2020, Haegeman_2016, schmidt2026} or rank-adaptive methods for time-dependent tensor decompositions \cite{ceruti2020}.

Finally, we touch upon the Gr\"obner bases of defining ideals of tensor trains. We know remarkably little even in this special case. We point out that the determinantal ideal generated by the minors of two flattening matrices in the case of tensor trains of order $3$ is a double determinantal ideal. These minors form a Gr\"obner basis with respect to a diagonal term order as proved in \cite{FK20} and \cite{illian2025grobner}. We construct an explicit diagonal term order for the general case and conjecture in Conjecture \ref{conj: minors-form-GB} that with respect to this term order the minors generating the tensor train ideal form a Gr\"obner basis. We finish with presenting a purely combinatorial algorithm to compute the degree of a tensor train of order $3$ (Proposition \ref{prop:degree-of-order-3}).  
 
The paper is structured as follows. Section \ref{sec:TTN} is devoted to the definition and description of the tree tensor network variety $\mathrm{TTN}_{\mathscr{T}, \mathbf{V}, \mathbf{r}}$. In particular, we give an explicit paramaterization of 
$\mathrm{TTN}_{\mathscr{T}, \mathbf{V}, \mathbf{r}}$ in 
Proposition \ref{prop:contraction parametrization}. In Section \ref{sec:tensor train varieties} and \ref{sec:subspace varieties} we rigorously prove that both tensor train and subspace representation varieties are indeed special cases of tree tensor networks. Section \ref{sec:spaced trees} is a review of spaced trees and general Markov models on spaced trees from \cite{DK09}, adjusted to our setting. We give the details of the recursive parametrization of these varieties (Definition \ref{defn:spaced-tree-contraction}) and also present a direct parametrization in Proposition \ref{prop:spaced-tree-parametrization}. Section \ref{sec:equations_DK} is devoted to the development of the defining ideal of the general Markov model on a spaced tree (Theorem \ref{thm:ideal}, which is Theorem 1.7 in \cite{DK09}). 
Section \ref{sec:TTN_are_spaced_trees} proves Theorem \ref{thm:TTN=Spaced_tree} and Section \ref{sec:equations} gives the details of the proof of our main Theorem \ref{thm: main}. Section \ref{sec: GB} treats   Gr\"obner bases and gives
 a combinatorial method to compute the degree of order $3$ tensor trains. 

\section{Tree tensor network varieties}\label{sec:TTN}
In this section we define the varieties of tree tensor networks intrinsically, then introduce corresponding parametrizations of these varieties via tensor contraction and finally specialize the derivation to prominent cases of TTN: the subspace \cite{Breiding_2024, Landsberg2012} and tensor train varieties \cite{BFHP25,Oseledets_2011}.

Let $\mathscr T$ be a \emph{combinatorial  rooted tree} with a set of $N$  \emph{leaves} $\mathscr L$ and a set of \emph{internal nodes} $\mathscr N$ including the \emph{root} $\rho$ such that $\mathscr{T} = \mathscr{L} \sqcup \mathscr{N}$.
Fix a set of $N$ vector spaces $V_\lambda = \CC^{d_\lambda}$ with $d_\lambda \in \NN$, sometimes called the \emph{physical dimension}, for $\lambda \in \mathscr L$. 

In order to define TTN in full generality, we consider each node $v \in \mathscr N$ to have immediate \emph{children} $\mathrm{ch}(v) \subseteq \mathscr T\setminus \{\rho\}$, that is, nodes $u$ connected to $v$ via edges pointing away from the root $\rho$.  Note that we do not require internal nodes to have exactly two children,  generalizing TTN on binary trees defined in \cite{BUCZYNSKA_2015}. 
In the following, we also assume that $|\mathrm{ch}(\rho)| \geq 2$. For every non-root vertex
$v\in\mathscr T\setminus\{\rho\}$, let
$\operatorname{par}(v)$ denote its unique \emph{parent}. We associate a
positive integer $r_v$ to the edge $\{\operatorname{par}(v),v\}$.
These \emph{ranks} are indexed by the non-root vertices $v$, and we write $\mathbf r
=
(r_v)_{v\in\mathscr T\setminus\{\rho\}}$ for the collection of ranks, called a \emph{rank sequence}.
For notational convenience, we set $r_\rho:=1$,
but $r_\rho$ is not part of the rank sequence $\mathbf r$.
We call the rank sequence $\mathbf r$ \emph{admissible} if $r_\lambda\leq d_\lambda$ for every $\lambda\in\mathscr L$ and $r_v
\leq
\prod_{u\in\operatorname{ch}(v)}r_u$ for every $v\in\mathscr N\setminus\{\rho\}$.
Given a tree $\mathscr T$, the vector spaces $\mathbf{V} = (V_\lambda)_{\lambda \in \mathscr L}$ and an admissible rank sequence $\mathbf{r}$, the variety associated to this tree tensor network is the projective tree tensor network variety
\begin{equation*}
    \mathrm{TTN}_{\mathscr T, \mathbf{V}, \mathbf{r}} \subset \PP(W),\quad W:= \bigotimes_{\lambda \in \mathscr{L}} V_\lambda.
\end{equation*}
\begin{definition}\label{def:subspace_def}
    The tree tensor network variety $\mathrm{TTN}_{\mathscr T, \mathbf{V}, \mathbf{r}}$ is the set of tensors $[\psi] \in \PP(W)$, such that for each vertex $v \in \mathscr T$ there exists a linear affine subspace $U_v$ with $\dim(U_v) \leq r_v$ and
    \begin{itemize}
        \item $U_\lambda \subseteq V_\lambda$ for $\lambda \in  \mathscr L$,
        \item $U_v \subseteq \bigotimes_{u \in \mathrm{ch}(v)}U_{u}$ for $v\in \mathscr N$, and
        \item there exists a representative $\psi \in U_{\rho}$. 
    \end{itemize}
\end{definition}
This definition generalizes \cite[Definition 2.5]{BUCZYNSKA_2015} to 
arbitrary trees. 
The next proposition justifies the fact that $\mathrm{TTN}_{\mathscr T, \mathbf{V}, \mathbf{r}}$ is a projective variety by realizing the $\mathrm{TTN}$ variety as the zero locus of minors of various flattenings of non-zero tensors $\psi$ for $ [\psi] \in \PP(W)$. 
For $v\in \mathscr T$, define $\mathscr L(v)$
as the set of leaves in the subtree rooted at $v$, and similarly $\mathscr{N}(v)$ and $\mathscr{T}(v)$. We set 
\[
W_v:=\bigotimes_{\lambda\in\mathscr L(v)}V_\lambda,
\qquad
W_v':=\bigotimes_{\lambda \in \mathscr{L} \setminus\mathscr L(v)}V_\lambda,
\]
such that the flattening for $v \in \mathscr T \setminus \{\rho\}$ reads:
\begin{equation} \label{eq:flattening}
    \psi^{(v)}: \left(W_v' \right)^* \to  W_v.
\end{equation}
In general, $\psi^{(v)}$ will be
a matrix of size $\left(\prod_{\lambda \in \mathscr{L}(v)} \dim(V_\lambda)\right) ~\times~\left(\prod_{\lambda \in \mathscr{L} \setminus \mathscr{L}(v)} \dim(V_\lambda)\right)$.

\begin{proposition} \label{prop:TTN_set_theoretic}
    The variety $\mathrm{TTN}_{\mathscr T, \mathbf{V}, \mathbf{r}}$ is set-theoretically given as
    \begin{equation}\label{eq:set_theoretic_TTN}
        \mathrm{TTN}_{\mathscr T, \mathbf{V}, \mathbf{r}} = \{[\psi] \in \PP(W) \, | \, \rank(\psi^{(v)}) \leq r_v ~ \text{for all} ~ v \in \mathscr T \setminus \{\rho\}\}.
    \end{equation}
\end{proposition}
\begin{proof}
We first show that Definition \ref{def:subspace_def} implies the description in \eqref{eq:set_theoretic_TTN}. 
Choose a nonzero representative $\psi\in W$ of $[\psi]\in
\mathrm{TTN}_{\mathscr T,\mathbf V,\mathbf r}$. By iterating the
nesting conditions in Definition \ref{def:subspace_def}, we obtain $\psi\in U_v\otimes W_v'$ for every $v\neq\rho$. Using \eqref{eq:flattening}, we get $\operatorname{im}(\psi^{(v)})\subseteq U_v$,
and therefore
$$
\operatorname{rank}(\psi^{(v)})
\leq \dim(U_v)
\leq r_v.
$$
For the converse, suppose that $\operatorname{rank}(\psi^{(v)})\leq r_v$
and define $U_v:=\operatorname{im}(\psi^{(v)})$ for $v\neq\rho$,
and set $U_\rho:=\operatorname{span}_{\mathbb C}\{\psi\}$.
Then $\dim(U_v)\leq r_v$ for every $v\neq\rho$, and
$U_\lambda\subseteq V_\lambda$ for every leaf $\lambda$.
To verify the nesting conditions of Definition \ref{def:subspace_def}, let
$v\in\mathscr N \setminus \{\rho\}$ be a non-leaf node and 
$\operatorname{ch}(v)=\{u_1,\ldots,u_s\}$.
For every $i \in [s]$, the definition
$U_{u_i}=\operatorname{im}(\psi^{(u_i)})$ implies
\[
\psi\in
U_{u_i}
\otimes
\bigotimes_{\substack{1\leq j\leq s\\j\neq i}}W_{u_j}
\otimes W_v'.
\]
Since the sets $\mathscr L(u_1),\ldots,\mathscr L(u_s)$ are pairwise
disjoint and their union is $\mathscr L(v)$, we have
\begin{equation*}
\bigcap_{i=1}^s
\left(
U_{u_i}
\otimes
\bigotimes_{\substack{1\leq j\leq s\\j\neq i}}W_{u_j}
\otimes W_v'
\right) 
=
\left(\bigotimes_{i=1}^sU_{u_i}\right)\otimes W_v'.
\end{equation*}
Thus
\[
\psi\in
\left(\bigotimes_{u\in\operatorname{ch}(v)}U_u\right)
\otimes W_v' \quad \text{and} \quad U_v=\operatorname{im}(\psi^{(v)})
\subseteq
\bigotimes_{u\in\operatorname{ch}(v)}U_u.
\]
At the root, the same intersection argument without the factor
$W_v'$ gives
$U_\rho\subseteq
\bigotimes_{u\in\operatorname{ch}(\rho)}U_u$.
This concludes the proof.
\end{proof}
We next describe a parametrization of $\mathrm{TTN}_{\mathscr T, \mathbf{V}, \mathbf{r}}$.
We denote by $\widehat{\mathrm{TTN}}_{\mathscr T,\mathbf V,\mathbf r}\subseteq W$ the affine cone over $\mathrm{TTN}_{\mathscr T,\mathbf V,\mathbf r}$.
For every $v\in\mathscr T\setminus\{\rho\}$, we associate the \emph{bond space} $E_v:=\mathbb C^{r_v}$ to the edge $\{\mathrm{par}(v), v\}$. At the root we set $E_\rho:=\mathbb C$. We give an example of a TTN with associated bond spaces in Figure \ref{fig:TTN_combinatorial_tree}.
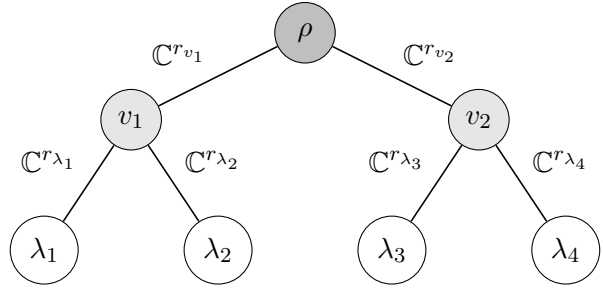
\begin{figure}[h!]
    \centering
    \begin{tikzpicture}[
        scale=1.15,
        every node/.style={font=\small},
        leaf/.style={circle, draw, minimum size=8mm},
        root/.style={circle, draw, fill=gray!50, minimum size=8mm},
        node/.style={circle, draw, fill=gray!20, minimum size=8mm},
        edge/.style={line width=0.6pt},
    ]

    \node[root] (r) at (0,4) {$\rho$};
    
    \node[node] (v11) at (-2,3) {$v_1$};
    \node[node] (v12) at ( 2,3) {$v_2$};
    
    \node[leaf] (l21) at (-3,1.5) {$\lambda_1$};
    \node[leaf] (l22) at (-1,1.5) {$\lambda_2$};
    \node[leaf] (l23) at (1,1.5) {$\lambda_3$};
    \node[leaf] (l24) at (3,1.5) {$\lambda_4$};
    
    \draw[edge] (r) -- node[midway, above left] {$\CC^{r_{v_1}}$}(v11);
    \draw[edge] (r) -- node[midway, above right, fill=white] {$\CC^{r_{v_2}}$}(v12);
    
    \draw[edge] (v11) -- node[midway, above left] {$\CC^{r_{\lambda_1}}$}(l21);
    \draw[edge] (v11) -- node[midway, above right] {$\CC^{r_{\lambda_2}}$}(l22);
    \draw[edge] (v12) -- node[midway, above left] {$\CC^{r_{\lambda_3}}$}(l23);
    \draw[edge] (v12) -- node[midway, above right] {$\CC^{r_{\lambda_4}}$}(l24);

    \end{tikzpicture}
    \caption{Binary TTN of order 4 with associated bond spaces.}
    \label{fig:TTN_combinatorial_tree}
\end{figure}
Furthermore, each vertex $v\in\mathscr T$ carries an associated \emph{local parameter space}
\begin{equation*}
        \mathcal A_v :=
    \begin{cases}
        V_\lambda\otimes E_\lambda^*,
            & v=\lambda\in\mathscr L,\\[2mm]
\left(\bigotimes_{u\in\operatorname{ch}(v)}E_u \right) \otimes E_v^*,
            & v\in\mathscr N \setminus \{\rho\},\\[2mm]
        \bigotimes_{u\in\operatorname{ch}(\rho)}E_u, & v = \rho.          
    \end{cases}
\end{equation*}
The total parameter space of the tree tensor network is
\begin{equation*}
        \mathcal P_{\mathscr T,\mathbf V,\mathbf r}
    :=
    \prod_{v\in\mathscr T}\mathcal A_v.
\end{equation*}
An element $\mathbf A:=(A^v)_{v\in\mathscr T}\in \mathcal P_{\mathscr T,\mathbf V,\mathbf r}$ of this space is called a \emph{collection of local cores}. The local core $A^v$ 
is a tensor expressed with respect to a choice of bases. Let $\{e^\lambda_{j_\lambda}\}_{j_\lambda=1}^{d_\lambda}$ and $\{f^v_{\alpha_v}\}_{\alpha_v=1}^{r_v}$ be bases of $V_\lambda$ and $E_v$, respectively.
We write
\begin{equation*}
        A^\lambda
    =
    \sum_{j_\lambda=1}^{d_\lambda}\sum_{\alpha_\lambda = 1}^{r_\lambda}
    A^\lambda_{j_\lambda,\alpha_\lambda}
    e^\lambda_{j_\lambda}\otimes(f^\lambda_{\alpha_\lambda})^*
\end{equation*}
for the matrices that are leaf cores,
\begin{equation*}
        A^v
    =
    \sum_{\alpha_v=1}^{r_v}\sum_{(\alpha_u)_{u\in\operatorname{ch}(v)}}
    A^v_{(\alpha_u),\alpha_v}
    \left(
    \bigotimes_{u\in\operatorname{ch}(v)}f^u_{\alpha_u}
    \right)
    \otimes(f^v_{\alpha_v})^*
\end{equation*}
for the tensors that are cores at nodes $v\in\mathscr N\setminus\{\rho\}$. We use the shorthand notation $(\alpha_u)_{u\in\operatorname{ch}(v)} = \prod_{u\in\operatorname{ch}(v)} [r_u]$ with $[r_u] = \{1, ..., r_u\}$. For the root we have
\begin{equation*}
        A^\rho
    =
    \sum_{(\alpha_u)_{u\in\operatorname{ch}(\rho)}}
    A^\rho_{(\alpha_u)}
    \bigotimes_{u\in\operatorname{ch}(\rho)}f^u_{\alpha_u}.
\end{equation*}
\begin{definition}
    The \emph{affine contraction map}
\begin{equation*}
        \widehat{\Phi}_{\mathscr T,\mathbf V,\mathbf r}:
    \mathcal P_{\mathscr T,\mathbf V,\mathbf r}
    \longrightarrow W 
\end{equation*}
is given as follows. 
Let $\widehat{\Phi}_\lambda \, : \, \mathcal{A}_\lambda \longrightarrow  V_\lambda \otimes E_\lambda^*$ 
be the identity map for each leaf $\lambda \in \mathscr{L}$. Then recursively define
\begin{equation} \label{eq:TTN_map_at_v}
\widehat{\Phi}_v \, : \,  \mathcal{A}_v \, \, \times  \,\, \prod_{u \in \operatorname{ch}(v)} W_u \otimes E_u^* \,\,  \longrightarrow \,\, W_v \otimes E_v^* 
\end{equation}
\begin{equation}\label{eq:Phi_v_formula}
\left(A^v,(B^u)_{u\in\operatorname{ch}(v)}\right)
\longmapsto
\sum_{(\alpha_u)_{u \in \operatorname{ch}(v)}}
\left(
\bigotimes_{u\in\operatorname{ch}(v)}
B^u_{\bullet,\alpha_u}
\right)
\otimes
A^v_{(\alpha_u),\bullet}.
\end{equation}
where $A^v_{(\alpha_u),\bullet}$ is the slice of the tensor $A^v$ determined by the fixed set of indices  $(\alpha_u)_{u \in \operatorname{ch}(v)}$ and $B^u_{\bullet, \alpha_u}$ is the slice of the tensor $B^u$ for the fixed index $\alpha_u$. 
Each entry of the resulting tensor is equal to 
\begin{align} \label{eq:TTN_param}
\left(\widehat{\Phi}_\rho(\mathbf{A})\right)_{(j_\lambda)_{\lambda\in\mathscr L}} &= \left(
    \widehat{\Phi}_{\mathscr T,\mathbf V,\mathbf r}
    (\mathbf A)
    \right)_{(j_\lambda)_{\lambda\in\mathscr L}} \nonumber\\
    &=
    \sum_{(\alpha_w)_{w\in\mathscr T\setminus\{\rho\}}}
    A^\rho_{(\alpha_u)_{u\in\operatorname{ch}(\rho)}}
    \prod_{v\in\mathscr N\setminus\{\rho\}}
    A^v_{(\alpha_u)_{u\in\operatorname{ch}(v)},\alpha_v}
    \prod_{\lambda\in\mathscr L}
    A^\lambda_{j_\lambda,\alpha_\lambda}.
\end{align}
Since the contraction is linear in each local core, it induces a projective
rational map
\begin{equation*}
    \Phi_{\mathscr T,\mathbf V,\mathbf r}:
    \prod_{v\in\mathscr T}\PP(\mathcal A_v)
    \dashrightarrow
    \PP(W),
    \qquad
    ([A^v])_{v\in\mathscr T}
    \longmapsto
    [\widehat{\Phi}_{\mathscr T,\mathbf V,\mathbf r}
    (\mathbf A)],
\end{equation*}
defined on the locus where the contraction is nonzero.
\end{definition}
\begin{example} \label{ex:main_TTN}
 We illustrate $\widehat{\Phi}_{\mathscr{T}, \mathbf{V}, \mathbf{r}}$ on the tree tensor network depicted in Figure 
 \ref{fig:TTN_combinatorial_tree}.
 We denote $\dim(V_{\lambda_i})=d_{\lambda_i} = d_i$. With this $A^{\lambda_i}$ is a $d_i \times r_{\lambda_i}$ matrix. A local core for the node $v_1$ is the $r_{\lambda_1}\times r_{\lambda_2} \times r_{v_1}$ tensor
 $A^{v_1}$. Similarly, a local core for the node $v_2$ is the $r_{\lambda_3}\times r_{\lambda_4} \times r_{v_2}$ tensor
 $A^{v_2}$. 
 A local core for the root $\rho$ is a $r_{v_1} \times r_{v_2}$ matrix $A^{\rho}$. Besides the four local cores at the leaves
 we obtain the two tensors $B^{v_1} \in \CC^{d_1} \otimes \CC^{d_2} \otimes \CC^{r_{v_1}}$ and $B^{v_2} \in \CC^{d_3} \otimes \CC^{d_4} \otimes \CC^{r_{v_2}}$ given by
 \begin{align*}
     (\widehat{\Phi}_{v_1}(A^{v_1}, A^{\lambda_1}, A^{\lambda_2}))_{j_{\lambda_1},j_{\lambda_2},\beta_1}=:B^{v_1}_{j_{\lambda_1},j_{\lambda_2},\beta_1} &= \sum_{\alpha_1=1}^{r_{\lambda_1}}\sum_{\alpha_2=1}^{r_{\lambda_2}} A^{v_1}_{\alpha_1,\alpha_2, \beta_1} A^{\lambda_1}_{j_{\lambda_1},\alpha_1} A^{\lambda_2}_{j_{\lambda_2},\alpha_2}, \\
 (\widehat{\Phi}_{v_2}(A^{v_2}, A^{\lambda_3}, A^{\lambda_4}))_{j_{\lambda_3},j_{\lambda_4},\beta_2}=:B^{v_2}_{j_{\lambda_3},j_{\lambda_4},\beta_2} &= \sum_{\alpha_3=1}^{r_{\lambda_3}}\sum_{\alpha_4=1}^{r_{\lambda_4}} A^{v_2}_{\alpha_3,\alpha_4, \beta_2} A^{\lambda_3}_{j_{\lambda_3},\alpha_3} A^{\lambda_4}_{j_{\lambda_4},\alpha_4},
 \end{align*}
 where we denote $\alpha_i = \alpha_{\lambda_i}$ and $\beta_i = \alpha_{v_i}$ to simplify notation.
 Finally, at the root we compute the tree tensor $\psi = \widehat{\Phi}_{\mathscr{T}, \mathbf{V}, \mathbf{r}}(\mathbf{A}) = \widehat{\Phi}_\rho(A^\rho, B^{v_1}, B^{v_2})$ where
 \begin{align}\label{eq:example_TTN_N_4}
     \psi_{j_{\lambda_1},j_{\lambda_2},j_{\lambda_3},j_{\lambda_4}} &= \sum_{\beta_1 =1}^{r_{v_1}}\sum_{\beta_2 =1}^{r_{v_2}} A^\rho_{\beta_1,\beta_2}B^{v_1}_{j_{\lambda_1},j_{\lambda_2},\beta_1}B^{v_2}_{j_{\lambda_3},j_{\lambda_4},\beta_2} \nonumber\\
     &= \sum_{\substack{\beta_1, \beta_2 \\ \alpha_1, \alpha_2, \alpha_3, \alpha_4}} A^\rho_{\beta_1,\beta_2}A^{v_1}_{\alpha_1,\alpha_2, \beta_1}A^{v_2}_{\alpha_3,\alpha_4, \beta_2} A^{\lambda_1}_{j_{\lambda_1},\alpha_1} A^{\lambda_2}_{j_{\lambda_2},\alpha_2}A^{\lambda_3}_{j_{\lambda_3},\alpha_3} A^{\lambda_4}_{j_{\lambda_4},\alpha_4}.
 \end{align} 
\end{example}
\begin{proposition}[Contraction parametrization]
\label{prop:contraction parametrization}
The tree tensor network variety
$\mathrm{TTN}_{\mathscr{T}, \mathbf{V}, \mathbf{r}}$ is equal to the image of 
$\Phi_{\mathscr{T}, \mathbf{V}, \mathbf{r}}$.
\end{proposition}
\begin{proof}
It is enough to show that the affine cone over the tree tensor network variety is equal to the image of
the affine contraction map, that is, 
\begin{equation*}
    \widehat{\mathrm{TTN}}_{\mathscr T,\mathbf V,\mathbf r}
    =
    \operatorname{im}
    \left(
    \widehat{\Phi}_{\mathscr T,\mathbf V,\mathbf r}
    \right).
\end{equation*}

First we show the inclusion $\operatorname{im}
    (
    \widehat{\Phi}_{\mathscr T,\mathbf V,\mathbf r}
    ) \subseteq \widehat{\mathrm{TTN}}_{\mathscr T,\mathbf V,\mathbf r}$.
 For non-zero $\psi \in \operatorname{im}(\widehat{\Phi}_{\mathscr T,\mathbf V,\mathbf r})$, pick a vertex $v \in \mathscr T \setminus \{\rho\}$. By removing the edge that connects $v$ to its parent, we get a subtree $\mathscr{T}(v)$ rooted at $v$.  Contracting all cores in the subtree rooted at $v$ gives a tensor $B^v \in W_v \otimes E_v^*$ (see Example \ref{ex:main_TTN}).
The entries $B^v_{\mathbf j_v, \alpha_v}$
of $B^v$ where $\mathbf j_v  = (j_\lambda)_{\lambda \in \mathscr L(v)}$
are computed via \eqref{eq:TTN_map_at_v} as
\begin{equation*}
    B^v_{\mathbf j_v, \alpha_v} = \sum_{(\alpha_w)_{w \in \mathscr{T}(v)}}
    A^v_{(\alpha_u)_{u\in\operatorname{ch}(v)}, \alpha_v}
    \prod_{w\in\mathscr N(v) \setminus\{v\}}
    A^w_{(\alpha_u)_{u\in\operatorname{ch}(w)},\alpha_w}
    \prod_{\lambda\in\mathscr L(v)}
    A^\lambda_{j_\lambda,\alpha_\lambda}.
\end{equation*}
Next, contract the cores in the complement of the subtree $\mathscr{T}(v)$, that is, all vertices $w\notin \mathscr{T}(v)$. Similarly to $B^v$, this gives a tensor $C^v \in E_v\otimes W_v'$,
such that $\psi_{\mathbf j_v, \mathbf j_v'} = \sum_{\alpha_v=1}^{r_v}B^v_{\mathbf j_v, \alpha_v}C^v_{\alpha_v, \mathbf j_v'} $ with $\mathbf j_v'  = (j_\lambda)_{\lambda \in \mathscr{L} \setminus \mathscr{L}(v)}$.
Therefore, the flattening $\psi^{(v)}$ can equivalently be written as a matrix product $\psi^{(v)} = B^vC^v$ and clearly $\mathrm{rank}(\psi^{(v)}) \leq r_v$.
Since this holds for all vertices $v\in\mathscr T \setminus \{\rho\}$ we conclude that 
$\operatorname{im}(\widehat{\Phi}_{\mathscr T,\mathbf V,\mathbf r}) \subseteq \widehat{\mathrm{TTN}}_{\mathscr T,\mathbf V,\mathbf r}$.

Now assume that $\psi \in \widehat{\mathrm{TTN}}_{\mathscr T,\mathbf V,\mathbf r}$ and recall Definition \ref{def:subspace_def}. For each $v\in\mathscr T\setminus\{\rho\}$, let $s_v:=\dim(U_v)\leq r_v.$
Choose a basis $\{u_1^v,\ldots,u_{s_v}^v\}$
of $U_v$, and set $u_\alpha^v:=0$ for $s_v<\alpha\leq r_v$.
Thus $(u_\alpha^v)_{\alpha=1}^{r_v}$ is a zero-padded spanning
family of $U_v$. 
For $\lambda \in \mathscr{L}$ we have $U_\lambda \subseteq V_\lambda$ and express the chosen basis of $U_\lambda$ in terms of the basis of $V_\lambda$,
    \begin{equation}\label{eq:basis_exp_lambda}
        u^\lambda_{\alpha_\lambda} = \sum_{j_\lambda=1}^{d_\lambda}A^\lambda_{j_\lambda, \alpha_\lambda}e^\lambda_{j_\lambda},
    \end{equation}
    where we have $A^\lambda_{j_\lambda, \alpha_\lambda} = 0$ whenever $s_\lambda < \alpha_\lambda \leq r_\lambda$.
    This defines the leaf core $A^\lambda \in V_\lambda \otimes E_\lambda^*$. 
    Next, for $v \in \mathscr N \setminus \{\rho\}$ we have  $U_v \subseteq \bigotimes_{u\in \mathrm{ch}(v)}U_u$. This allows us to write
    \begin{equation}\label{eq:basis_exp_v}
        u^v_{\alpha_v} = \sum_{(\alpha_u)_{u\in\mathrm{ch}(v)}} A^v_{(\alpha_u),\alpha_v}\bigotimes_{u\in\mathrm{ch}(v)}u^{u}_{\alpha_{u}},
    \end{equation}
    which defines the core $A^v \in \left(\bigotimes_{u\in\operatorname{ch}(v)}E_u \right) \otimes E_v^*$, where again $A^v_{(\alpha_u),\alpha_v} = 0$ whenever $s_v < \alpha_v \leq r_v$.
    Finally, since $ \psi\in U_\rho \subseteq    \bigotimes_{u\in\operatorname{ch}(\rho)}U_u$, we obtain the root core 
    \begin{equation}\label{eq:psi_subspace_proof}
    \psi
    =
    \sum_{(\alpha_u)_{u\in\operatorname{ch}(\rho)}}
    A^\rho_{(\alpha_u)_{u\in\operatorname{ch}(\rho)}}
    \bigotimes_{u\in\operatorname{ch}(\rho)}
    u^u_{\alpha_u}.
    \end{equation}
 Taking into account \eqref{eq:basis_exp_lambda} and \eqref{eq:basis_exp_v}, we see that $\psi$ in \eqref{eq:psi_subspace_proof} can be expressed in the basis of $W$ and this expression is precisely the right-hand side of \eqref{eq:TTN_param}. This shows that 
    $\psi \in 
    \operatorname{im} (\widehat{\Phi}_{\mathscr T,\mathbf V,\mathbf r})$  and concludes the proof.   
\end{proof}
\noindent
We now illustrate the construction of general TTN using tensor train (TT) and subspace varieties. Both are prominent examples of varieties arising from tensor decompositions via sequential application of singular value or QR decompositions.
\subsection{Tensor train varieties} \label{sec:tensor train varieties}
Tensor trains, also known as matrix product states \cite{Schollwoeck2011}, are a special case of tree tensor networks. Here we define them and illustrate how one can view them as tree tensor networks.

\begin{definition}
A tensor train variety is determined by two sequences of positive
integers
\[
\mathbf d=(d_1,\ldots,d_N)
\qquad\text{and}\qquad
\mathbf r=(r_0=1,r_1,\ldots,r_{N-1},r_N=1).
\]
For each $i=1,\ldots,N$, let $A^i
\in
\CC^{d_i}
\otimes
\mathbb C^{r_i}
\otimes
(\mathbb C^{r_{i-1}})^*$.
Since $r_0=r_N=1$, the corresponding one-dimensional factors in
the first and last cores may be omitted.
The tensor train variety $\mathrm{TT}_{\mathbf d,\mathbf r}$ is
the Zariski closure of the image of the rational map \cite[Section 5]{BFHP25}
\begin{equation}\label{eq:tensor-train-param}
\begin{aligned}
\Psi_{\mathbf d,\mathbf r}:
\prod_{i=1}^N
\mathbb P\left(
\CC^{d_i}
\otimes
\mathbb C^{r_i}
\otimes
(\mathbb C^{r_{i-1}})^*
\right)
&\dashrightarrow
\mathbb P\left(
\bigotimes_{i=1}^N \CC^{d_i}
\right),\\
([A^i])_{i=1}^N
&\longmapsto
[\psi],
\end{aligned}
\end{equation}
where
\begin{equation}\label{eq:TT-entries}
\psi_{j_1,\cdots ,j_N}
=
\sum_{\alpha_1=1}^{r_1}
\cdots
\sum_{\alpha_{N-1}=1}^{r_{N-1}}
A^1_{j_1,\alpha_1}
A^2_{j_2,\alpha_2,\alpha_1}
\cdots
A^N_{j_N,\alpha_{N-1}}.
\end{equation}
\end{definition}

\begin{proposition}\label{prop:TT_is_TTN}
Let $\mathbf{d}=(d_1, \ldots, d_N)$ and $\mathbf{r} = (r_0=1, r_1, \ldots, r_{N-1},r_N=1)$ be sequences of positive integers. Then the corresponding TT variety is equal to a TTN variety,
$$ \mathrm{TT}_{\mathbf{d}, \mathbf{r}} = \mathrm{TTN}_{\mathscr{T}, \mathbf{V},\mathbf{r}},$$
where 
\begin{enumerate}
\item $\mathscr{L} = \{\lambda_1,\ldots, \lambda_N\}$ and $\mathscr{N} = \{v_1=\rho, v_2, \ldots, v_N\}$, and the edges in $\mathscr{T}$ are  $\{\lambda_i, v_i\}$ 
for $i=1, \ldots, N$ and $\{v_i, v_{i+1}\}$ for $i=1,\ldots, N-1$,
\item $V_{\lambda_i} = \CC^{d_i}$ for $i=1,\ldots, N$, and
\item $r_{\lambda_i}=d_i$ for $i=1,\ldots,N$, so that
$E_{\lambda_i}=V_{\lambda_i}$, and
\[
r_{v_i}=r_{i-1}
\qquad
\text{for }i=2,\ldots,N.
\]
As before, we use the convention $r_{v_1}=r_\rho=1$.

\end{enumerate}
\end{proposition}
\begin{proof}
Under the rank sequence above, we have $E_{\lambda_i}=V_{\lambda_i}$ for $i \in [N]$ and $E_{v_i}\simeq\mathbb C^{r_{i-1}}$ for $i=2,\ldots,N$, together with $E_{v_1}=E_\rho=\mathbb C$.
First, let $(G^i)_{i=1}^N$ be a collection of ordinary tensor-train
cores, that is, a parametrization realizing a point in $\mathrm{TT}_{\mathbf{d}, \mathbf{r}}$. For every leaf $\lambda_i$, choose the TTN leaf core to be $A^{\lambda_i}
=
\operatorname{id}_{V_{\lambda_i}}.$
Choose the internal TTN core at $v_i$ to be $G^i$, using the
canonical identifications
$$
\begin{aligned}
\mathcal A_\rho
&=
E_{\lambda_1}\otimes E_{v_2}
\simeq
\CC^{d_1}\otimes\mathbb C^{r_1},\\
\mathcal A_{v_i}
&=
E_{\lambda_i}\otimes E_{v_{i+1}}\otimes E_{v_i}^*
\simeq
\CC^{d_i}
\otimes
\mathbb C^{r_i}
\otimes
(\mathbb C^{r_{i-1}})^* \quad \quad i=2,\ldots, N-1,\\
\mathcal A_{v_N}
&=
E_{\lambda_N}\otimes E_{v_N}^*
\simeq
\CC^{d_N}\otimes(\mathbb C^{r_{N-1}})^*.
\end{aligned}
$$
Hence, we realize the tensor-train cores as TTN cores with $G^1 \in \mathcal A_\rho$, $G^i \in \mathcal A_{v_i}$ for $i \in \{2,\ldots, N-1\}$ and $G^N \in \mathcal A_{v_N}$.
Then the TTN contraction \eqref{eq:TTN_param} reduces to the
tensor-train contraction \eqref{eq:TT-entries} and we conclude 
$
\mathrm{TT}_{\mathbf d,\mathbf r}
\subseteq
\mathrm{TTN}_{\mathscr T,\mathbf V,\mathbf r}$.

Conversely, consider TTN cores $(A^v)_{v\in\mathscr{T}}$ that realize a point on $\mathrm{TTN}_{\mathscr T,\mathbf V,\mathbf r}$. Write
$
L^i:=A^{\lambda_i}
\in
V_{\lambda_i}\otimes E_{\lambda_i}^*$
for the leaf core at $\lambda_i$, and write 
$ B^i:=A^{v_i} \in E_{\lambda_i}\otimes E_{v_{i+1}}\otimes E_{v_i}^*
$
for the adjacent internal core. Contract the
$E_{\lambda_i}^*$-factor of $L^i$ with the
$E_{\lambda_i}$-factor of $B^i$ for $i \in [N]$, and define a 
tensor-train core $G^i$ by
\[
G^i_{j_{\lambda_i},\alpha_{v_i},\alpha_{v_{i-1}}}
:=
\sum_{\alpha_{\lambda_i}=1}^{d_i}
L^i_{j_{\lambda_i},\alpha_{\lambda_i}}
B^i_{\alpha_{\lambda_i},\alpha_{v_i},\alpha_{v_{i-1}}} \in
\CC^{d_i}
\otimes
\mathbb C^{r_i}
\otimes
(\mathbb C^{r_{i-1}})^*,
\]
where we use the endpoint conventions $\alpha_0=\alpha_N=1.$
Setting $j_i = j_{\lambda_i}$, $\alpha_i = \alpha_{v_i}$ and substituting these $\mathrm{TT}$ cores into
\eqref{eq:TT-entries} gives exactly the tensor produced by the
original TTN contraction \eqref{eq:TTN_param} for $\mathscr{T}$ defined above. Hence
$
\mathrm{TTN}_{\mathscr T,\mathbf V,\mathbf r}
\subseteq
\mathrm{TT}_{\mathbf d,\mathbf r}$, and
the two varieties are therefore equal.
\end{proof}

\begin{example}
Let $\mathbf{d} = (d_1, d_2,d_3)$ and 
$\mathbf{r} = (r_0=1, r_1, r_2, r_3=1)$
define a tensor train variety in the space of tree tensors in $\CC^{d_1} \otimes \CC^{d_2} \otimes \CC^{d_3}$. 
The corresponding tree $\mathscr{T}$
that realizes this tensor train variety as a TTN can be seen in Figure \ref{fig:TTN=TT_tree}. The local parameter spaces of the TTN are
\begin{equation*}
    \mathcal A_{\lambda_i} = V_{\lambda_i}\otimes V_{\lambda_i}^*, \,\, \mathcal A_{v_1} = V_{\lambda_1}\otimes E_{v_2} \otimes \CC^*, \,\, A_{v_2} = V_{\lambda_2}\otimes E_{v_3} \otimes E_{v_2}^* , \,\, 
    A_{v_3} = V_{\lambda_3}\otimes \CC \otimes E_{v_3}^*.
\end{equation*}
In this case the parametrization reads
\begin{align*}
\widehat{\Phi}_{\mathscr{T}, \mathbf{V}, \mathbf{r}}(\mathbf{A})     &= \sum_{\alpha_{v_2}=1}^{r_{1}}\sum_{\alpha_{v_3}=1}^{r_{2}}\prod_{i = 1}^{3}\sum_{j_{\lambda_i}=1}^{d_i} A^1_{j_{\lambda_1}, \alpha_{v_2}} A^2_{j_{\lambda_2}, \alpha_{v_3}, \alpha_{v_2}}A^3_{j_{\lambda_3}, \alpha_{v_3}} e_{j_{\lambda_1}}\otimes e_{j_{\lambda_2}} \otimes e_{j_{\lambda_3}}. 
\end{align*}
\end{example}
\begin{figure}[h!]
    \centering
    \begin{tikzpicture}[
        scale=1.15,
        every node/.style={font=\small},
        leaf/.style={circle, draw, minimum size=8mm},
        root/.style={circle, draw, fill=gray!50, minimum size=8mm},
        node/.style={circle, draw, fill=gray!20, minimum size=8mm},
        edge/.style={line width=0.6pt},
    ]

    \node[node] (a1) at (0,4) {$v_1$};
    
    \node[leaf] (l1) at (-1,3) {$\lambda_1$};
    \node[node] (a2) at (1,3)  {$v_2$};

    \node[leaf] (l2) at (0, 2) {$\lambda_2$};
    \node[node] (a3) at (2, 2) {$v_3$};

    \node[leaf] (l3) at (1, 1) {$\lambda_3$};

    \draw[edge] (a1) -- node[midway, above left] {$E_{\lambda_1}$}(l1);
    \draw[edge] (a1) -- node[midway, above right, fill=white] {$E_{v_2}$}(a2);
    \draw[edge] (a2) -- node[midway, above left] {$E_{\lambda_2}$}(l2);
    \draw[edge] (a2) -- node[midway, above right] {$E_{v_3}$}(a3);
    \draw[edge] (a3) -- node[midway, above left] {$E_{\lambda_3}$}(l3);

    \end{tikzpicture}
    \caption{Tree tensor network of a tensor train decomposition of order 3 with bond spaces associated to the edges.}
    \label{fig:TTN=TT_tree}
\end{figure}
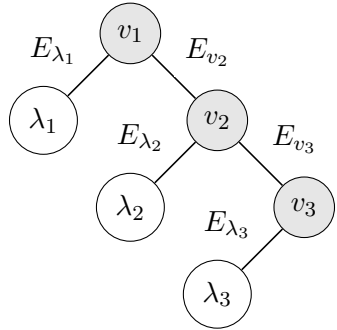
\begin{corollary}\label{cor: TT flattening}
Let $\mathbf{d}=(d_1, \ldots, d_N)$ and $\mathbf{r} = (r_0=1, r_1, \ldots, r_{N-1},r_N=1)$ be sequences of positive integers. The tensor train variety $\mathrm{TT}_{\mathbf{d}, \mathbf{r}}$ is
$$ \mathrm{TT}_{\mathbf{d}, \mathbf{r}} = \{ [\psi] \in \PP(W) \, : \, \mathrm{rank}(\psi^{(i)}) \leq r_i, \,\, i=1, \ldots, N-1\}$$
where $\psi^{(i)}$ is the flattening of 
$\psi$ corresponding to the partition 
$\{1,\ldots,i\} \mid\{i+1, \ldots, N\}$.
\end{corollary}
\begin{proof}
Combining Propositions \ref{prop:TTN_set_theoretic} and  
\ref{prop:TT_is_TTN} we get
$$ \mathrm{TT}_{\mathbf{d}, \mathbf{r}} =\mathrm{TTN}_{\mathscr T,\mathbf V,\mathbf r} =  \{ [\psi] \in \PP(W) \, : \, \mathrm{rank}(\psi^{(v_i)}) \leq r_{v_i}, \,\, i=2, \ldots, N\}.$$
For $i=2,\ldots,N$, the flattening $\psi^{(v_i)}$ in \eqref{eq:flattening} corresponds to
the partition $\{i,\ldots,N\}\mid\{1,\ldots,i-1\}.$
It is therefore the transpose of $\psi^{(i-1)}$ and $\operatorname{rank}(\psi^{(v_i)})
=
\operatorname{rank}(\psi^{(i-1)}).$
Since $r_{v_i}=r_{i-1}$, the result follows.
\end{proof}
\subsection{Subspace varieties}
\label{sec:subspace varieties}
Another class of tree tensor network varieties are subspace varieties (see \cite{Landsberg2012}; for ideal-theoretic description see \cite{landsberg2007, oeding2016}). As in the tensor train case, we show that subspace varieties are a special case of tree tensor network varieties. 
\begin{definition} 
A subspace variety is determined by two sequences of positive
integers
\[
\mathbf d=(d_1,\ldots,d_N)
\qquad\text{and}\qquad
\mathbf r=(r_1,\ldots,r_N).
\]
These determine local cores $A^i \in \CC^{d_i} \otimes (\CC^{r_i})^*$ for $i=1, \ldots, N$, and $C \in \bigotimes_{i=1}^{N} \CC^{r_i}$. The subspace variety $\mathrm{S}_{\mathbf{d}, \mathbf{r}}$ is the image closure of the 
rational map
\begin{equation}\label{eq:subspace-param}
    \begin{aligned}
        \Psi_{\mathbf{d},\mathbf{r}}:
\left(\prod_{i=1}^N \PP\left(\CC^{d_i} \otimes (\CC^{r_i})^*\right)\right) \times \PP(\bigotimes_{i=1}^{N} \CC^{r_i} )  
& \dashrightarrow
\mathbb{P}\!\left(\bigotimes_{i=1}^N \mathbb{C}^{d_i}\right), \\
\big(([A^i])_{i=1}^N, [C]\big)
& \longmapsto
[\psi] ,
    \end{aligned}
\end{equation}
where 
\begin{equation} \label{eq:S-entries}
 \psi_{j_1, \cdots, j_N} = \sum_{\alpha_1=1}^{r_1}
\cdots
\sum_{\alpha_{N}=1}^{r_{N}} C_{\alpha_1, ..., \alpha_N} A^1_{j_1,\alpha_1} A^2_{j_2,\alpha_2} \cdots A^N_{j_N, \alpha_{N}}.
 \end{equation}

\end{definition}
\begin{proposition}\label{prop:S_is_TTN}
Let $\mathbf d=(d_1,\ldots,d_N)$ and $\mathbf{r} = (r_1, \ldots,r_N)$ be sequences of positive integers. Then the corresponding subspace variety is equal to a TTN variety,
$$ \mathrm{S}_{\mathbf{d}, \mathbf{r}} = \mathrm{TTN}_{\mathscr{T}, \mathbf{V},\mathbf{r}},$$
where 
\begin{enumerate}
\item $\mathscr{L} = \{\lambda_1,\ldots, \lambda_N\}$ and $\mathscr{N} = \{\rho\}$, and the edges in $\mathscr{T}$ are  $\{\lambda_i, \rho\}$ 
for $i=1, \ldots, N$,
\item $V_{\lambda_i} = \CC^{d_i}$ and $r_{\lambda_i} = r_i$ for $i=1,\ldots, N$. 
\end{enumerate}
\end{proposition}
\begin{proof}
For each leaf $\lambda_i$ the local core $A^{\lambda_i}$  is the $d_i \times r_{\lambda_i}$ matrix  in $\mathcal A_{\lambda_i} = V_{\lambda_i}\otimes E_{\lambda_i}^*$. Since $\mathscr{N} = \{\rho\}$, there are no other internal nodes other than the root. For the root, we have the core $A^\rho \in \bigotimes_{i=1}^{N} E_{\lambda_i}$. Renaming
$A^i = A^{\lambda_i}$ and $C = A^\rho$, the tree tensor network parametrization in \eqref{eq:TTN_param} reduces to the subspace parametrization in \eqref{eq:S-entries}.
\end{proof}
\begin{example}
Let $\mathbf d=(d_1,d_2, d_3)$ and 
$\mathbf{r} = (r_1, r_2, r_3)$
define a subspace variety in the space of tree tensors in $\CC^{d_1} \otimes \CC^{d_2} \otimes \CC^{d_3}$. 
The corresponding tree $\mathscr{T}$
that realizes this subspace variety as a TTN can be seen in Figure \ref{fig:subspace_combinatorial_tree}. The local TTN parameter spaces are
\begin{equation*}
    \mathcal A_{\lambda_i} = V_{\lambda_i}\otimes E_{\lambda_i}^*, \,\, A^\rho \in E_{\lambda_1} \otimes E_{\lambda_2} \otimes E_{\lambda_3}.
\end{equation*}
In this case the parametrization reads 
\begin{align*}
\widehat{\Phi}_{\mathscr{T}, \mathbf{V}, \mathbf{r}}(\mathbf{A})     &= \sum_{\alpha_{\lambda_1}=1}^{r_1}\sum_{\alpha_{\lambda_2}=1}^{r_2}
\sum_{\alpha_{\lambda_3}=1}^{r_3}\prod_{i=1}^{3}\sum_{j_{\lambda_i}=1}^{d_i} C_{\alpha_{\lambda_1}, \alpha_{\lambda_2}, \alpha_{\lambda_3}}A^1_{j_{\lambda_1}, \alpha_{\lambda_1}} A^2_{j_{\lambda_2}, \alpha_{\lambda_2}}A^3_{j_{\lambda_3}, \alpha_{\lambda_3}} e_{j_{\lambda_1}}\otimes e_{j_{\lambda_2}} \otimes e_{j_{\lambda_3}}. 
\end{align*}
\end{example}
\begin{figure}[h!]
    \centering
    \begin{tikzpicture}[
        scale=1.15,
        every node/.style={font=\small},
        leaf/.style={circle, draw, minimum size=8mm},
        root/.style={circle, draw, fill=gray!50, minimum size=8mm},
        node/.style={circle, draw, fill=gray!20, minimum size=8mm},
        edge/.style={line width=0.6pt},
    ]

    \node[node] (rho) at (0,0) {$\rho$};
    
    \node[leaf] (l1) at (-2,-1.5) {$\lambda_1$};
    \node[leaf] (l2) at (0, -1.5) {$\lambda_2$};
    \node[leaf] (l3) at (2, -1.5) {$\lambda_3$};

    \draw[edge] (rho) -- node[midway, above left] {$E_{\lambda_1}$}(l1);
    \draw[edge] (rho) -- node[pos=0.9, above left] {$E_{\lambda_2}$}(l2);
    \draw[edge] (rho) -- node[midway, above right] {$E_{\lambda_3}$}(l3);

    \end{tikzpicture}
    \caption{Tree tensor network of a subspace (Tucker) decomposition of order 3 with bond spaces associated to the edges.}
    \label{fig:subspace_combinatorial_tree}
\end{figure}
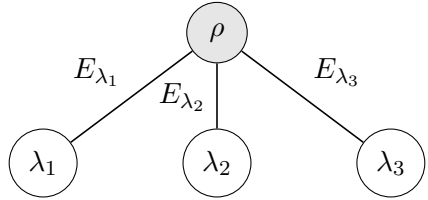
\begin{corollary}\label{cor: S flattening}
Let $\mathbf d=(d_1,\ldots,d_N)$ and $\mathbf{r} = (r_1, \ldots,r_N)$ be sequences of positive integers. The subspace variety $\mathrm{S}_{\mathbf{d}, \mathbf{r}}$ is
 $$ \mathrm{S}_{\mathbf{d}, \mathbf{r}} = \{ [\psi] \in \PP(W) \, :\, \mathrm{rank}(\psi^{(i)}) \leq r_i \mbox{ for } i=1,\ldots, N\}$$
 where $\psi^{(i)}$ is the flattening of $\psi$ corresponding to the partition $\{i\} \mid \{1, \ldots, \widehat{i}, \ldots, N\}$.
\end{corollary}
\begin{proof}
Combining Propositions \ref{prop:TTN_set_theoretic} and  
\ref{prop:S_is_TTN} we get
$$ \mathrm{S}_{\mathbf{d}, \mathbf{r}} = \mathrm{TTN}_{\mathscr{T}, \mathbf{V},\mathbf{r}} = \{ [\psi] \in \PP(W) \, : \, \mathrm{rank}(\psi^{(\lambda_i)}) \leq r_{\lambda_i}, \,\, i=1, \ldots, N\}.
$$
For $i \in [N]$, the flattening $\psi^{(\lambda_i)}$ corresponds to the partition $\{i\} \mid \{1, \ldots, \widehat{i}, \ldots, N\}$. Since $r_{\lambda_i} = r_i$, the result follows.
\end{proof}
\section{General Markov models on spaced trees} \label{sec:spaced trees}
Our aim in the next section is to realize any given tree tensor network variety $\mathrm{TTN}_{\mathscr T, \mathbf{V}, \mathbf{r}}$ as a {\it general Markov model} associated to a {\it spaced tree}. Spaced trees and Markov models on them were treated in \cite{DK09} as a general framework for phylogenetic trees representing biological evolution via nucleotide mutations. In this section we present a concise summary of what we need from \cite{DK09} for the remaining sections.

\begin{definition}\cite[Definition 1.2.]{DK09}
A spaced tree $T$ is 
a finite undirected tree with a complex vector space $V_v$ of dimension $\mathrm{dim}(V_v) = n_v$ attached to each vertex $v$. Moreover, for each vertex $v$, we fix a non-degenerate symmetric bilinear form $\langle\cdot,\cdot \rangle_v: V_v \times V_v \rightarrow \CC$ and a basis $B_v$ that is orthonormal with respect to $\langle \cdot, \cdot \rangle_v$.
\end{definition}
The set of leaves of a spaced tree $T$ is denoted by $\mathcal{L}(T)$. In this setting, we define a leaf as a node $v\in T$ that has only one edge attached to it. The set of remaining vertices in $T$, that is $v\in T\setminus \mathcal{L}(T)$, is defined as \emph{internal nodes} $\mathrm{int}(T)$. The ambient tensor space associated to $T$ is $\bigotimes_{v \in \mathcal{L}(T)} V_v =: K_T$ and we denote the edges of $T$ by $\mathcal{E}(T)$. 
\begin{definition}\cite[Definition 1.4.]{DK09}
A \emph{representation} of a spaced tree $T$ assigns to every edge $e=\{v,w\}$ a tensor $E_e \in V_v \otimes V_w$. The set of all representations of the spaced tree $T$ is $\mathrm{Rep}(T)= \prod _{e=\{v,w\}} (V_v \otimes V_w)$. 
\end{definition}
Once we fix the orthonormal bases $B_v$ for each vertex $v$, the tensors $E_e$ with $e = \{v,w\}$ will be identified with matrices $A_e$ inducing linear homomorphisms $V_v \longrightarrow V_w$. With this we will view $\mathrm{Rep}(T)$ as a parameter space where the entries of these matrices are parameters. 

\begin{example} The simplest nontrivial example is the \emph{star tree} $T$ with three leaves: $T$ consists of a single internal vertex $q \in \mathrm{int}(T)$ and three leaves $\lambda_1, \lambda_2, \lambda_3 \in \mathcal{L}(T)$. If we assume $V_q \simeq \CC^{n_q}$
and $V_{\lambda_i} \simeq \CC^{n_{\lambda_i}}$ for some $n_q, n_{\lambda_i} \in \NN$ and $i=1,2,3$, 
then 
$$ \mathrm{Rep}(T) \, =\, \{A = (A_{\lambda_1 q}, A_{\lambda_2 q}, A_{\lambda_3 q}) \, : \, A_{\lambda_i q} \in \mathrm{Mat}_{ n_{\lambda_i} \times n_q}(\CC), \,\, i=1,2,3 \} $$
where $\mathrm{Mat}_{ n_{\lambda_i} \times n_q}(\CC)$ is the space of all complex $n_{\lambda_i} \times n_q$ matrices. 
\end{example}

\begin{definition}\cite{DK09} \label{defn:spaced-tree-contraction}
The \emph{polynomial contraction map}
\begin{equation*}
    \Psi_T \, : \, \mathrm{Rep}(T) \longrightarrow K_T
\end{equation*}
associated to a spaced tree $T$ is defined recursively. If $T$ has a single edge $e=\{v,w\}$ with $V_v \simeq \CC^{n_v}$
and $V_w \simeq \CC^{n_w}$ and a representation $A \in \mathrm{Mat}_{n_v \times n_w}$, then $\Psi_T(A) = A \in \CC^{n_v} \otimes \CC^{n_w}$. 
If $T$ has at least two edges, pick an internal node $q \in \mathrm{int}(T)$. After deleting $q$ and edges incident to it from $T$ one obtains $s\geq 2$ subtrees $T_i'$. Let $v_i$ be the vertex of $T_i'$ where the deleted edge was incident. Now, let $T_i$ for $i=1, \ldots, s$ be the \emph{subtree} obtained by re-attaching to $T_i'$ this unique deleted edge $\{v_i,q\}$ at $v_i$ (including the node $q$).
 Given a representation $A \in \mathrm{Rep}(T)$, let $A_i \in \mathrm{Rep}(T_i)$ be the induced representation on each $T_i$. For these subtrees let
\begin{equation}
\label{eq:subtree-rep-map}
\Psi_{T_i} \, : \, \mathrm{Rep}(T_i) \longrightarrow \bigotimes_{v \in \mathcal{L}(T_i)} V_v =: K_{T_i}
\end{equation}
 be the associated recursively defined maps. We decompose $\Psi_T(A)$ via these subtree maps $\Psi_{T_i}(A_i)$,
\begin{equation}\label{eq:psi_map_explicit}
    \Psi_T(A) = \sum_{j_q=1}^{n_q} \left( \Psi_{T_1}(A_1)^{j_q} \otimes \Psi_{T_2}(A_2)^{j_q} \otimes \cdots \otimes \Psi_{T_s}(A_s)^{j_q}\right) \in K_T, 
\end{equation}
where $\Psi_{T_i}(A_i)^{j_q}$ is the slice of the tensor $\Psi_{T_i}(A_i)$ with fixed index $j_q$ corresponding to the node $q$. We refer to the contraction defined above as the \emph{gluing} of subtrees at node $q$. 
\\
We now prove a statement that highlights the fact that $\Psi_T(A)$ can be constructed by contracting  the matrices associated to all edges $\mathcal{E}(T)$.  For this define $\mathcal{L}(T_i)$ as the set of leaves of subtree $T_i$ (which includes node $q$).
\end{definition}

\begin{proposition} \label{prop:spaced-tree-parametrization} Let $T$ be a spaced tree and $A = (A_e \, : \, e \in \mathcal{E}(T)) \in \mathrm{Rep}(T)$. Then
\begin{equation} \label{eq:spaced-tree-param} 
(\Psi_T(A))_{(j_v)_{v \in \mathcal{L}(T)}} = \sum_{(j_v)_{v\in T \setminus \mathcal{L}(T)}} \prod_{\{u,w\} \in \mathcal{E}(T)} (A_{uw})_{j_u j_w}.
\end{equation}
\end{proposition}
\begin{proof}
We use induction based on the recursion that defines $\Psi_T$. In this recursion, when a $T$ is just an edge $\{u,w\}$ with $\mathrm{Rep}(T) = \{A_{uw}\}$, then $(\Psi_T(A_{uw}))_{j_uj_w} = (A_{uw})_{j_uj_w}$, and in this simple case, this is equal to the expression in \eqref{eq:spaced-tree-param} since there is no non-leaf to sum over and there is exactly one edge. For the induction step, let $q$ be an internal node of $T$ and let $T_i$, $i=1,\ldots, s$, be the subtrees adjacent to $q$ with their respective maps $\Psi_{T_i}$ as in  \eqref{eq:subtree-rep-map}. Moreover, $\mathrm{Rep}(T_i)$ consists of $A_i = (A_e \, : \, e \in \mathcal{E}(T_i))$.
By induction assumption
\begin{equation}\label{eq:subtree_edge_param}
    (\Psi_{T_i}(A_i))_{(j_v)_{v \in \mathcal{L}(T_i)}} = \sum_{(j_v)_{v \in T_i \setminus\mathcal{L}(T_i)}} \prod_{\{u,w\} \in \mathcal{E}(T_i)} (A_{uw})_{j_u j_w}.
\end{equation}
Note that in the product above we have matrices $A_{v_iq}$, where $\{v_i, q\}$ is the edge in $T_i$ that is incident to 
the leaf $q$.  Now $\Psi_T(A) = \Psi_T(A_i \, : \, i=1, \ldots, s)$ 
and using the contraction map 
for $\Psi_T(A)$ as in \eqref{eq:psi_map_explicit}, we get
\begin{equation}\label{eq:full_tree_contraction}
    (\Psi_T(A))_{(j_v)_{v \in \mathcal{L}(T)}} = \sum_{j_q=1}^{n_q} \prod_{i=1}^s ((\Psi_{T_i}(A_i))_{(j_v)_{v \in \mathcal{L}(T_i)}})^{j_{q}}.
\end{equation}
Substituting $(\Psi_{T_i}(A_i))_{(j_v)_{v \in \mathcal{L}(T_i)}}$ in \eqref{eq:subtree_edge_param} in expression \eqref{eq:full_tree_contraction} yields \eqref{eq:spaced-tree-param}.
\end{proof}

\begin{definition}\cite[Definition 1.5]{DK09}
The \emph{affine general Markov model} associated to $T$ is
\[
\widehat{\V}_T
:=
\overline{\operatorname{im}(\Psi_T)}
\subseteq
K_T,
\]
where the closure is taken in the Zariski topology. Since
$\operatorname{Rep}(T)$ is an affine space and $\Psi_T$ is a
polynomial map, $\widehat{\V}_T$ is irreducible. We call $\V_T:=\mathbb P(\widehat{\V}_T)$ the corresponding irreducible projective spaced-tree variety.
\end{definition}

\begin{example} \label{ex:sp-tree}
Consider the following tree $T$. 
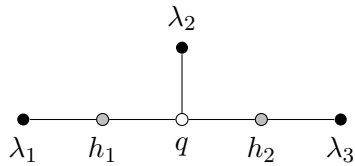
\begin{center}
\begin{tikzpicture}[
  scale=1,
  transform shape,
  x=1.0cm,
  y=1.0cm,
  every node/.style={font=\small},
  leaf/.style={circle, fill, inner sep=1.6pt},
  internal/.style={circle, draw, inner sep=1.6pt},
  rank/.style={circle, draw, fill=gray!50, inner sep=1.6pt},
  edge/.style={line width=0.6pt}
]

\def\dx{1.05}  
\def\dy{0.95}  

\node[leaf, label=below:$\lambda_1$] (L1) at (-\dx,0) {};
\node[rank, label=below:$h_1$] (h1) at (0,0) {};
\node[internal, label=below:$q$] (q) at (\dx,0) {};
\node[rank, label=below:$h_2$] (h2) at (2*\dx,0) {};
\node[leaf, label=below:$\lambda_3$] (l3) at (3*\dx,0) {};

\node[leaf, label=above:$\lambda_2$] (L2) at (\dx,\dy) {};
\draw (q) -- (L2);

\draw (L1) -- (h1) -- (q) -- (h2) -- (l3);

\end{tikzpicture}
\end{center}
Let $V_{\lambda_1} = V_{\lambda_2} = V_{\lambda_3} \simeq \CC^3$, $V_{h_1} = V_{h_2} \simeq \CC^2$,  and $V_q \simeq \CC^4$.
A representation consists of 
$$A = (A_{\lambda_1 h_1}, A_{h_1 q}, A_{\lambda_2 q}, A_{h_2 q}, A_{\lambda_3 h_2})$$
where $A_{\lambda_1 h_1} = (a_{j_{\lambda_1}, j_{h_1}}) \in \mathrm{Mat}_{3 \times 2}(\CC)$, $A_{\lambda_2 q} = (b_{j_{\lambda_2},j_{q}}) \in \mathrm{Mat}_{3 \times 4}(\CC)$, $A_{\lambda_3 h_2} = (c_{j_{\lambda_3}, j_{h_2}}) \in \mathrm{Mat}_{3 \times 2}(\CC)$, $A_{h_1 q} = (d_{j_{h_1},j_{q}}) \in \mathrm{Mat}_{2 \times 4}(\CC)$, and $A_{h_2 q} = (e_{j_{h_2}, j_{q}}) \in \mathrm{Mat}_{2 \times 4}(\CC)$.
We first split $T$ into two trees $T_1$  and $T_2$:
\begin{center}
\begin{tikzpicture}[
  scale=1,
  transform shape,
  x=1.0cm,
  y=1.0cm,
  every node/.style={font=\small},
  leaf/.style={circle, fill, inner sep=1.6pt},
  internal/.style={circle, draw, inner sep=1.6pt},
  rank/.style={circle, draw, fill=gray!50, inner sep=1.6pt},
  edge/.style={line width=0.6pt}
]

\def\dx{1.05}  
\def\dy{0.95}  

\node[leaf, label=below:$\lambda_1$] (L1) at (-\dx,0) {};
\node[rank, label=below:$h_1$] (h1) at (0,0) {};

\draw (L1) -- (h1);

\node at (\dx,0){};
\node[rank, label=below:$h_1$] (H1) at (2*\dx,0) {};
\node[internal, label=below:$q$] (q) at (3*\dx,0) {};
\node[rank, label=below:$h_2$] (H2) at (4*\dx,0) {};
\node[leaf, label=below:$\lambda_3$] (L3) at (5*\dx,0) {};

\draw (H1) -- (q) -- (H2) -- (L3);

\node[leaf, label=above:$\lambda_2$] (L2) at (3*\dx,\dy) {};

\draw (L2) -- (q);

\end{tikzpicture}
\end{center}
Next $T_2$ is split into two trees $\tilde{T}_1$ and $\tilde{T}_2$:
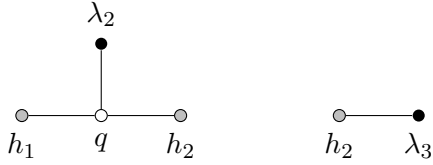
\begin{center}
\begin{tikzpicture}[
  scale=1,
  transform shape,
  x=1.0cm,
  y=1.0cm,
  every node/.style={font=\small},
  leaf/.style={circle, fill, inner sep=1.6pt},
  internal/.style={circle, draw, inner sep=1.6pt},
  rank/.style={circle, draw, fill=gray!50, inner sep=1.6pt},
  edge/.style={line width=0.6pt}
]

\def\dx{1.05}  
\def\dy{0.95}  

\node at (\dx,0){};
\node[rank, label=below:$h_1$] (H1) at (2*\dx,0) {};
\node[internal, label=below:$q$] (q) at (3*\dx,0) {};
\node[rank, label=below:$h_2$] (H2) at (4*\dx,0) {};

\draw (H1) -- (q) -- (H2);

\node[leaf, label=above:$\lambda_2$] (L2) at (3*\dx,\dy) {};

\draw (L2) -- (q);

\node[rank, label=below:$h_2$] (HH2) at (6*\dx,0) {};
\node[leaf, label=below:$\lambda_3$] (L3) at (7*\dx,0) {};

\draw (HH2) -- (L3);

\end{tikzpicture}
\end{center}
Finally, $\tilde{T}_1$ is split into three trees $\bar{T}_1, \bar{T}_2$, and $\bar{T}_3$:
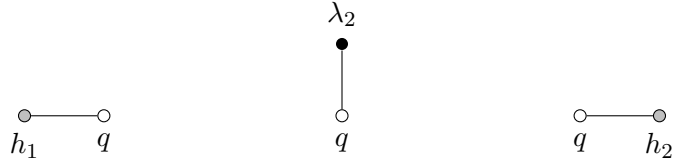
\begin{center}
\begin{tikzpicture}[
  scale=1,
  transform shape,
  x=1.0cm,
  y=1.0cm,
  every node/.style={font=\small},
  leaf/.style={circle, fill, inner sep=1.6pt},
  internal/.style={circle, draw, inner sep=1.6pt},
  rank/.style={circle, draw, fill=gray!50, inner sep=1.6pt},
  edge/.style={line width=0.6pt}
]

\def\dx{1.05}  
\def\dy{0.95}  

\node[rank, label=below:$h_1$] (H1) at (-\dx,0) {};
\node[internal, label=below:$q$] (q1) at (0,0) {};

\draw (H1) -- (q1);

\node at (\dx,0){};

\node[internal, label=below:$q$] (q) at (3*\dx,0) {};

\draw (q) ;

\node[leaf, label=above:$\lambda_2$] (L2) at (3*\dx,\dy) {};

\draw (L2) -- (q);

\node[internal, label=below:$q$] (q2) at (6*\dx,0) {};
\node[rank, label=below:$h_2$] (H2) at (7*\dx,0) {};

\draw (q2) -- (H2);

\end{tikzpicture}
\end{center}
We get $\Psi_{\bar{T}_1}(A_{h_1q}) = A_{h_1q}$,  $\Psi_{\bar{T}_2}(A_{\lambda_2q}) = A_{\lambda_2q}$, and $\Psi_{\bar{T}_3}(A_{h_2q}) = A_{h_2q}$. We glue these trees together at $q$ via \eqref{eq:psi_map_explicit} into $\tilde{T}_1$ where
 \begin{equation*}
     P := \Psi_{\tilde{T}_1}(A_{h_1 q}, A_{\lambda_2 q}, A_{h_2 q}) = \sum_{j_q=1}^{n_q = 4}(A_{h_1 q})^{j_q}\otimes (A_{\lambda_2 q})^{j_q}\otimes (A_{h_2 q})^{j_q} \in \CC^2 \otimes \CC^3 \otimes \CC^2,
 \end{equation*}
 and $P_{j_{h_1}, j_{\lambda_2}, j_{h_2}} = \sum_{j_q=1}^{4}d_{j_{h_1}, j_q}b_{j_{\lambda_2}, j_q}e_{j_{h_2}, j_q}$.
Next, we glue $\tilde{T}_1$ and $\tilde{T}_2$ at $h_2$ into $T_2$. Noting that $\Psi_{\tilde{T}_2}(A_{\lambda_3h_2}) = A_{\lambda_3h_2}$, we get the tensor
\begin{equation*}
    Q := \Psi_{T_2}(A_{h_1 q}, A_{\lambda_2 q}, A_{h_2 q}, A_{\lambda_3h_2}) = \sum_{j_{h_2}=1}^{n_{h_2}=2}(\Psi_{\tilde{T}_1}(A_{h_1 q}, A_{\lambda_2 q}, A_{h_2 q}))^{j_{h_2}}\otimes(\Psi_{\tilde{T}_2}(A_{\lambda_3h_2}))^{j_{h_2}}
\end{equation*}
in $\CC^2 \otimes \CC^3 \otimes \CC^3$ given by
\begin{equation*}
    Q_{j_{h_1}, j_{\lambda_2}, j_{\lambda_3}} = \sum_{j_{h_2}=1}^2 P_{j_{h_1}, j_{\lambda_2}, j_{h_2}}c_{j_{\lambda_3}, j_{h_2}} = \sum_{j_{h_2}=1}^2 \sum_{j_q=1}^4 d_{j_{h_1}, j_q}b_{j_{\lambda_2}, j_q}e_{j_{h_2}, j_q}c_{j_{\lambda_3}, j_{h_2}}.
\end{equation*}
Finally, we glue $T_1$ and $T_2$ at $h_1$ into $T$. Again, since $\Psi_{T_1}(A_{\lambda_1h_1}) = A_{\lambda_1 h_1}$, the result is a tensor $R \in \CC^3 \otimes \CC^3 \otimes \CC^3$ with 
\begin{align*}
    R := \Psi_{T}(A) = \sum_{j_{h_1} = 1}^{n_{h_1}=2}(\Psi_{T_2}(A_{h_1 q}, A_{\lambda_2 q}, A_{h_2 q}, A_{\lambda_3h_2}))^{j_{h_1}}\otimes (\Psi_{T_1}(A_{\lambda_1h_1}))^{j_{h_1}},
\end{align*}
given by 
\begin{equation*}
    R_{j_{\lambda_1}, j_{\lambda_2}, j_{\lambda_3}} = \sum_{j_{h_1} = 1}^{2} a_{j_{\lambda_1}, j_{h_1}}Q_{j_{h_1}, j_{\lambda_2}, j_{\lambda_3}} =  \sum_{j_{h_1} = 1}^{2}\sum_{j_{h_2}=1}^2 \sum_{j_q=1}^4 a_{j_{\lambda_1}, j_{h_1}}d_{j_{h_1}, j_q}b_{j_{\lambda_2}, j_q}e_{j_{h_2}, j_q}c_{j_{\lambda_3}, j_{h_2}}.
\end{equation*}
In summary we get the map 
$$ \Psi_T \, : \, \mathrm{Rep}(T) \longrightarrow \CC^3 \otimes \CC^3 \otimes \CC^3$$
given by
$$ (\Psi_T(A))_{j_{\lambda_1}, j_{\lambda_2}, j_{\lambda_3}} = \sum_{j_{h_1} = 1}^{2}\sum_{j_{h_2}=1}^2 \sum_{j_q=1}^4 a_{j_{\lambda_1}, j_{h_1}}d_{j_{h_1}, j_q}b_{j_{\lambda_2}, j_q}e_{j_{h_2}, j_q}c_{j_{\lambda_3}, j_{h_2}}. $$
This expression is precisely the one claimed by Proposition \ref{prop:spaced-tree-parametrization} for the tree $T$ as above.
\end{example}

\subsection{Equations of spaced-tree varieties}\label{sec:equations_DK}
We continue to follow \cite{DK09} and now describe the ideal of $\V_T$.
For this, we need one more construction.
\begin{definition}\cite{DK09}\label{def:contracted_star_tree}
    Let $q$ be an internal vertex of $T$. Deleting $q$ and the incident edges produces finitely many trees, the \emph{connected components} of $T \setminus \{q\}$. We define an equivalence relation on $\mathcal{L}(T) \cup \{q\}$ by $p \sim r$ if and only if either $p=q=r$
or $p,r \neq q$ lie in the same connected component of 
 $T \setminus \{q\}$. 
This gives rise to the \emph{contracted star tree} $\mathfrak{s}_q(T)$ as follows. The vertices of this new tree are $\mathcal{L}(T) \cup \{q\}  / \sim$ and the class of $q$ is attached to all other classes by an edge. The resulting star tree $\mathfrak{s}_q(T)$ has (the class of) $q$ as its internal node and every other class corresponds to a leaf. We collect the classes $C$ associated to the leaves in the set $\mathcal{C}(q)$ and define the \emph{degree} of the node $q$ as $
D_q = |\mathcal{C}(q)|$. To each of these \emph{leaf classes} $C$ we associate the space $K_C := \bigotimes_{v \in C} V_v$. 
 The set of representations $\mathrm{Rep}(\mathfrak{s}_q(T))$ consists of collections of matrices $(A_{C, q} \, : \, C \in \mathcal{C}(q))$
where $A_{C,q} \in K  _C \otimes V_q^*$ is a $d_C \times n_q$ matrix with $d_C = \prod_{v \in C} n_v$.
\end{definition}

\begin{remark}\label{rmk:s_q_variety}
The image of the star-tree parametrization consists precisely of tensors of tensor rank at most $n_q$, that is, of tensors (see Proposition \ref{prop:spaced-tree-parametrization} and \eqref{eq:psi_map_explicit})
\begin{equation*}
    \psi = \Psi_{\mathfrak{s}_q(T)}(A)  = \sum_{j_q=1}^{n_q} \bigotimes_{C \in \mathcal{C}(q)} (A_{C,q})^{j_q} \in K_T,
\end{equation*}
where $(A_{C,q})^{j_q} \in K_C$ is the $j_q$th column of $A_{C,q}$. 
The corresponding spaced-tree variety $\widehat{\V}_{\mathfrak{s}_q(T)}$ is, by definition, the Zariski closure of this image. Consequently,
\begin{equation*}
    \V_{\mathfrak{s}_q(T)} = \sigma_{n_q}\left(\mathrm{Seg}\left(\bigtimes_{C \in \mathcal{C}(q)} \PP\left(K_C\right)\right)\right) \subseteq \PP\left( K_T\right),
\end{equation*}
namely, the $n_q$-th secant variety
of the Segre product of  projective spaces $\{ \PP\left(K_C\right) \, : \, C \in \mathcal{C}(q)\}$, which consists of projective tensors of border rank at most $n_q$. When $\mathfrak{s}_q(T)$ has only two leaf classes $C$ and $\overline{C} = \mathcal{L}(T) \setminus C$,
this secant variety is equal to all tensors $[\psi]$ whose flattenings corresponding to the partition $C\mid\overline{C}$ have rank at most $n_q$. In this case, the ideal $\mathcal{I}(\V_{\mathfrak{s}_q(T)})$ is generated by the $(n_q+1)$-minors of this flattening. When $\mathfrak{s}_q(T)$ has three or more leaf classes, there is no uniform determinantal description of the ideal in general. 
\end{remark}

\noindent
Now we are ready to present a fundamental theorem which describes the homogeneous prime ideal of the general Markov model $\mathcal{V}_T$ of the spaced tree $T$.
\begin{theorem} \label{thm:ideal} \cite[Theorem 1.7]{DK09} For any spaced tree $T$ we have 
$$ \mathcal{I}(\V_T) = \sum_{v \in \mathrm{int}(T)} \mathcal{I}(\V_{\mathfrak{s}_v(T)}).$$
\end{theorem}
\begin{example}
Consider the spaced tree $T$ in Example \ref{ex:sp-tree}. The contracted star trees for the three internal vertices $h_1, q, h_2$ are:
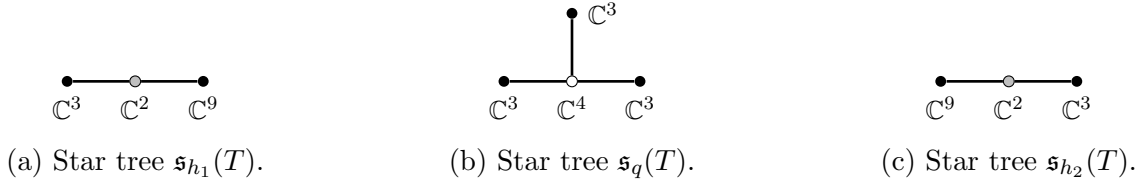
\begin{figure}[htbp]
\centering
\setlength{\abovecaptionskip}{3pt}
\setlength{\belowcaptionskip}{-3pt}
\begin{subfigure}[t]{0.30\textwidth}
\centering
\begin{tikzpicture}[
  scale=0.9,
  transform shape,
  x=1.0cm, y=1.0cm,
  every node/.style={font=\small},
  leaf/.style={circle, fill, inner sep=1.6pt},
  internal/.style={circle, draw, inner sep=1.6pt},
  rank/.style={circle, draw, fill=gray!50, inner sep=1.6pt},
  edge/.style={line width=0.6pt}
]

\node[leaf, label=below:$\CC^3$] (h11) at (-1,0) {};
\node[leaf, label=below:$\CC^9$] (h12) at (1,0) {};
\node[rank, label=below:$\CC^2$] (q0) at (0,0) {};

\draw[edge, line width=1pt] (h11) -- (q0);
\draw[edge, line width=1pt] (h12) -- (q0);

\end{tikzpicture}
\caption{Star tree $\mathfrak{s}_{h_1}(T)$.}
\label{fig:Tl1}
\end{subfigure}
\hfill
\begin{subfigure}[t]{0.30\textwidth}
\centering
\begin{tikzpicture}[
  scale=0.9,
  transform shape,
  x=1.0cm, y=1.0cm,
  every node/.style={font=\small},
  leaf/.style={circle, fill, inner sep=1.6pt},
  internal/.style={circle, draw, inner sep=1.6pt},
  rank/.style={circle, draw, fill=gray!50, inner sep=1.6pt},
  edge/.style={line width=0.6pt}
]

\node[leaf, label=below:$\CC^3$] (h21) at (-1,0) {};
\node[leaf, label=right:$\CC^3$] (h22) at (0,1) {};
\node[leaf, label=below:$\CC^3$] (h11) at (1,0) {};
\node[internal, label=below:$\CC^4$] (q11) at (0,0) {};

\draw[edge, line width=1pt] (h21) -- (q11);
\draw[edge, line width=1pt] (h22) -- (q11);
\draw[edge, line width=1pt] (h11) -- (q11);

\end{tikzpicture}
\caption{Star tree $\mathfrak{s}_{q}(T)$.}
\label{fig:Tq0}
\end{subfigure}
\hfill
\begin{subfigure}[t]{0.30\textwidth}
\centering
\begin{tikzpicture}[
  scale=0.9,
  transform shape,
  x=1.0cm, y=1.0cm,
  every node/.style={font=\small},
  leaf/.style={circle, fill, inner sep=1.6pt},
  internal/.style={circle, draw, inner sep=1.6pt},
  rank/.style={circle, draw, fill=gray!50, inner sep=1.6pt},
  edge/.style={line width=0.6pt}
]

\node[leaf, label=below:$\CC^9$] (h11) at (-1,0) {};
\node[leaf, label=below:$\CC^3$] (h12) at (1,0) {};
\node[rank, label=below:$\CC^2$] (q0) at (0,0) {};

\draw[edge, line width=1pt] (h11) -- (q0);
\draw[edge, line width=1pt] (h12) -- (q0);

\end{tikzpicture}
\caption{Star tree $\mathfrak{s}_{h_2}(T)$.}
\label{fig:T_q11}
\end{subfigure}

\vspace{0.5em}

\caption{The contracted star trees of the spaced tree $T$ in Example \ref{ex:sp-tree} with associated vector spaces at the vertices.}
\label{fig:TTN_spaced_tree_combined}
\end{figure}
\\
According to Theorem \ref{thm:ideal} \cite{DK09}
$$ \mathcal{I}(\V_T) = \mathcal{I}(\V_{\mathfrak{s}_{h_1}(T)}) + \mathcal{I}(\V_{\mathfrak{s}_q(T)}) + \mathcal{I}(\V_{\mathfrak{s}_{h_2}(T)}).$$
Here every ideal lies in the polynomial ring $\CC[\psi_{ijk} \, : \, i,j,k=1,2,3]$.
We note that $\mathfrak{s}_{h_1}(T)$ consists of a path  with two edges. The spaces attached to its leaves are 
$V_{\lambda_1} \simeq \CC^3$ and $\bigotimes_{i\in\{2, 3\}}V_{\lambda_i} \simeq \CC^9 $. The space attached to its internal vertex is $V_{h_1} \simeq \CC^2$. Therefore, $\V_{\mathfrak{s}_{h_1}(T)}$ 
is the set of all $3 \times 9$ matrices of rank at most $2$ whose $(i,jk)$ entry is $\psi_{ijk}$.  
Hence, $\mathcal{I}(\V_{\mathfrak{s}_{h_1}(T)})$
is generated by all $3$-minors of the generic $3 \times 9$ matrix which is a flattening of the tensor $\psi$. The flattening corresponds to the partition $1|23$. Similarly, $\mathcal{I}(\V_{\mathfrak{s}_{h_2}(T)})$ is 
generated by all $3$-minors of the generic $9 \times 3$ matrix which is the flattening of $\psi$ corresponding to the partition $12|3$.  The remaining star tree $\mathfrak{s}_q(T)$ has three leaves to each of which $\CC^3$ is attached. Since $V_q \simeq \CC^4$, the variety $\V_{\mathfrak{s}_q(T)}$ is the closure of all tensors in $\CC^3 \otimes \CC^3 \otimes \CC^3$ with tensor rank at most $4$. Therefore, $\mathcal{I}(\V_{\mathfrak{s}_q(T)})$ is the ideal of the secant variety $\V_{\mathfrak{s}_q(T)} = \sigma_4(\mathrm{Seg}(\PP^2 \times \PP^2 \times \PP^2))$.
\end{example}

\section{TTN varieties are Markov models on spaced trees}\label{sec:TTN_are_spaced_trees}
In this section, we prove that the 
tree tensor network variety $\mathrm{TTN}_{\mathscr{T}, \mathbf{V}, \mathbf{r}}$ is equal to the general Markov model $\V_T$ on a corresponding spaced tree $T$. We will prove this by showing that the image of the parametrization $\widehat{\Phi}_{\mathscr T,\mathbf V,\mathbf r}$ for a tree tensor network $\mathscr{T}$, vector spaces $\boldsymbol{V}$ and rank sequence $\boldsymbol{r}$ coincides with the image of the parametrization $\Psi_T$ associated to a spaced tree $T$. 

We first construct the tree $T$ from the given TTN $\mathscr{T}$: it is obtained from $\mathscr{T}$ by inserting a new vertex of degree two on each edge of $\mathscr{T}$. This new vertex $h$ is referred to as a \emph{rank node}. Hence every edge $\{v,u\}$ of $\mathscr{T}$ is replaced by two edges  $\{v,h\}$ and $\{h,u\}$ that are connected via the node $h$. We keep the labels of the original vertices coming from $\mathscr{T}$ and label the new vertices as follows. Suppose that for the edge $\{v, u\}$ the vertex $v$ is an internal node and $u$ is one of its children in $\mathscr{T}$. Then the new vertex $h$ inbetween $v$ and $u$ will be labeled as $h_u$. See Figure \ref{fig:TTN_vs_Markov_tree_model} for an example. To each $h_v$ we attach the complex vector space $V_{h_v} = E_v$ with $\mathrm{dim}(V_{h_v}) = r_v$ from $\boldsymbol{r}$. Finally, to each node $v$ 
that is not a leaf and not a rank node, we attach the complex vector space $V_v$ with $\mathrm{dim}(V_v) = n_v := (\prod_{u \in \mathrm{ch}(v)} r_u)r_v$. Note that the leaves $\mathcal{L}(T) =: \mathcal{L}$ of $T$ are precisely the TTN leaves
$\mathscr L$ of $\mathscr T$. For every
$\lambda\in\mathscr L=\mathcal L(T)$, we attach the physical space $V_\lambda=\mathbb C^{d_\lambda}$ from $\boldsymbol{V}$. We equip all vertex spaces with the standard symmetric bilinear forms and use the resulting identifications $E_v\simeq E_v^*$.

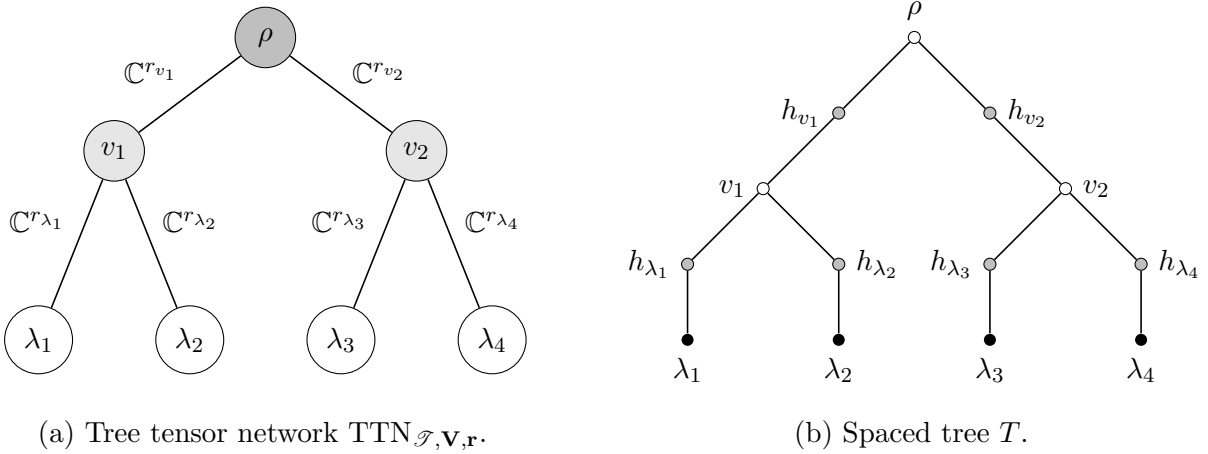
\begin{figure}[htbp]
\centering
\setlength{\abovecaptionskip}{3pt}
\setlength{\belowcaptionskip}{-3pt}

\begin{subfigure}[t]{0.48\textwidth}
\centering
\begin{tikzpicture}[
        scale=1,
        every node/.style={font=\small},
        leaf/.style={circle, draw, minimum size=8mm},
        root/.style={circle, draw, fill=gray!50, minimum size=8mm},
        node/.style={circle, draw, fill=gray!20, minimum size=8mm},
        edge/.style={line width=0.6pt},
    ]

    \node[root] (r) at (0,0) {$\rho$};

    \node[node] (v11) at (-2,-1.5) {$v_1$};
    \node[node] (v12) at ( 2,-1.5) {$v_2$};

    \node[leaf] (l21) at (-3,-4) {$\lambda_1$};
    \node[leaf] (l22) at (-1,-4) {$\lambda_2$};
    \node[leaf] (l23) at ( 1,-4) {$\lambda_3$};
    \node[leaf] (l24) at ( 3,-4) {$\lambda_4$};

    \draw[edge] (r) -- 
        node[midway, above left] {$\CC^{r_{v_1}}$} (v11);
    \draw[edge] (r) -- 
        node[midway, above right, fill=white] {$\CC^{r_{v_2}}$} (v12);

    \draw[edge] (v11) -- 
        node[midway, above left] {$\CC^{r_{\lambda_1}}$} (l21);
    \draw[edge] (v11) -- 
        node[midway, above right] {$\CC^{r_{\lambda_2}}$} (l22);
    \draw[edge] (v12) -- 
        node[midway, above left] {$\CC^{r_{\lambda_3}}$} (l23);
    \draw[edge] (v12) -- 
        node[midway, above right] {$\CC^{r_{\lambda_4}}$} (l24);

    \pgfresetboundingbox
    \path[use as bounding box] (-4,-4.8) rectangle (4,0.8);

\end{tikzpicture}
\caption{Tree tensor network $\mathrm{TTN}_{\mathscr{T}, \mathbf{V}, \mathbf{r}}$.}
\label{fig:tree_tensor_network}
\end{subfigure}
\hfill
\begin{subfigure}[t]{0.48\textwidth}
\centering
\begin{tikzpicture}[
  scale=1,
  transform shape,
  x=1.0cm,
  y=1.0cm,
  every node/.style={font=\small},
  leaf/.style={circle, fill, inner sep=1.6pt},
  internal/.style={circle, draw, inner sep=1.6pt},
  rank/.style={circle, draw, fill=gray!50, inner sep=1.6pt},
  edge/.style={line width=0.6pt}
]

\node[internal, label=above:$\rho$] (q) at (0,0) {};

\node[rank, label=left:$h_{v_1}$]  (h11) at (-1,-1) {};
\node[rank, label=right:$h_{v_2}$] (h12) at ( 1,-1) {};

\node[internal, label=left:$v_1$]  (q11) at (-2,-2) {};
\node[internal, label=right:$v_2$] (q12) at ( 2,-2) {};

\node[rank, label=left:$h_{\lambda_1}$]  (h21) at (-3,-3) {};
\node[rank, label=right:$h_{\lambda_2}$] (h22) at (-1,-3) {};
\node[rank, label=left:$h_{\lambda_3}$]  (h23) at ( 1,-3) {};
\node[rank, label=right:$h_{\lambda_4}$] (h24) at ( 3,-3) {};

\node[leaf, label=below:$\lambda_1$] (l1) at (-3,-4) {};
\node[leaf, label=below:$\lambda_2$] (l2) at (-1,-4) {};
\node[leaf, label=below:$\lambda_3$] (l3) at ( 1,-4) {};
\node[leaf, label=below:$\lambda_4$] (l4) at ( 3,-4) {};

\draw[edge] (q) -- (h11);
\draw[edge] (q) -- (h12);
\draw[edge] (q11) -- (h11);
\draw[edge] (q12) -- (h12);

\draw[edge] (q11) -- (h21);
\draw[edge] (q11) -- (h22);
\draw[edge] (q12) -- (h23);
\draw[edge] (q12) -- (h24);

\draw[edge] (l1) -- (h21);
\draw[edge] (l2) -- (h22);
\draw[edge] (l3) -- (h23);
\draw[edge] (l4) -- (h24);

\pgfresetboundingbox
\path[use as bounding box] (-4,-4.8) rectangle (4,0.8);

\end{tikzpicture}
\caption{Spaced tree $T$.}
\label{fig:Markov_spaced_tree}
\end{subfigure}

\vspace{0.5em}

\caption{Comparison of a TTN of order four and its corresponding spaced tree.}
\label{fig:TTN_vs_Markov_tree_model}
\end{figure}

Next we describe $\mathrm{Rep}(T)$, namely, the parameter space for the map $\Psi_T$. To each edge $\{\lambda, h_{\lambda}\}$ we attach $A^\lambda \in \mathcal{A}_\lambda \simeq \mathrm{Mat}_{d_\lambda \times r_\lambda}$ for $\lambda \in \mathcal{L}(T)$.  Otherwise, any other edge is of the form $\{h_u, v\}$, where $u$ is a child of internal node $v$ in $\mathscr{T}$ or it is $\{v, h_v\}$ (see Figure \ref{fig:Markov_spaced_tree}). We assign matrices $M^{h_u}  \in \mathrm{Mat}_{r_u \times n_v}$ to $\{h_u, v\}$ and matrices $M^v \in \mathrm{Mat}_{r_v \times n_v}$ to $\{v, h_v\}$. The parameter space $\mathrm{Rep}(T)$ is equal to the direct product of all these matrix spaces, one per edge $\mathcal{E}(T)$ of the tree $T$. 
\begin{example}
    The parameter space for the Markov model on spaced tree $T$ in Figure \ref{fig:Markov_spaced_tree} is
\begin{align*}
    \mathrm{Rep}(T) = \{A&=(A^{\lambda_1}, A^{\lambda_2}, A^{\lambda_3}, A^{\lambda_4}, M^{h_{\lambda_1}}, M^{h_{\lambda_2}}, M^{h_{\lambda_3}}, M^{h_{\lambda_4}}, M^{v_1}, M^{v_2}, M^{h_{v_1}}, M^{h_{v_2}})~:~\\
    &A^{\lambda_i} \in \mathrm{Mat}_{d_{\lambda_i} \times r_{\lambda_i}}, M^{h_{\lambda_1}}\in\mathrm{Mat}_{r_{\lambda_1} \times n_{v_1}}, M^{h_{\lambda_2}}\in\mathrm{Mat}_{r_{\lambda_2} \times n_{v_1}}, M^{h_{\lambda_3}}\in\mathrm{Mat}_{r_{\lambda_3}\times n_{v_2}}, \\
    &M^{h_{\lambda_4}}\in\mathrm{Mat}_{r_{\lambda_4} \times n_{v_2}}, M^{v_i}\in\mathrm{Mat}_{r_{v_i} \times n_{v_i}}, M^{h_{v_i}}\in\mathrm{Mat}_{r_{v_i} \times n_{\rho}}\}.
\end{align*}
\end{example}

\TTNisSpacedTree*
\begin{proof}
We prove the statement by showing that the image of $\widehat{\Phi}_{\mathscr{T}, \mathbf{V}, \mathbf{r}}$ is equal to the image of $\Psi_T$. For this, we will link the parameter spaces of both maps by realizing that
\begin{equation*}
    \mathrm{Rep}(T) \overset{\Theta}{\longrightarrow} \mathcal{P}_{\mathscr{T}, \mathbf{V}, \mathbf{r}} \overset{\widehat{\Phi}_{\mathscr{T}, \mathbf{V}, \mathbf{r}}}{\longrightarrow} W,
\end{equation*}
with $\Psi_T = \widehat{\Phi}_{\mathscr{T}, \mathbf{V}, \mathbf{r}} \circ \Theta$ and where $\Theta \, : \, \mathrm{Rep}(T) \longrightarrow \mathcal{P}_{\mathscr{T}, \mathbf{V}, \mathbf{r}}$ is surjective.
\\
To see this, we split $T$ into subtrees. These are obtained by cutting the tree $T$ at the  rank nodes $h_v$, such that the result is a collection of star trees $T_v$ for each $v \in \mathscr{T}$. The tree $T_\lambda$ for the leaf $\lambda \in \mathcal{L}$ is just the edge $\{\lambda, h_\lambda\}$. For a non-leaf non-rank node $v$, the tree $T_v$ is the star tree with $v$ as its unique internal node and the set of leaves $\mathcal{L}(T_v) = (h_u)_{u \in \mathrm{ch}(v)} \cup h_v$, which we call adjacent rank nodes. We note that when $v=\rho$ the leaf $h_\rho$ does not exist. Observe that
for each non-rank node $v$ the parameter space $\mathrm{Rep}(T_v)$ consists of the set of matrices that are associated to the edges included in the subtree $T_v$. The edge sets $\mathcal{E}(T_v)$ are pairwise disjoint and their union is $\mathcal{E}(T)$. Therefore, we can write $\mathrm{Rep}(T) \simeq \prod_{v\in \mathscr{T}}\mathrm{Rep}(T_v)$. With this factorization it is natural to define $\Theta:=\prod_{v \in \mathscr{T}} \Psi_{T_v}$. The image of $\Psi_{T_\lambda}$ consists of $A^\lambda \in \mathrm{Mat}_{ d_\lambda\times r_\lambda} \simeq \mathcal{A}_\lambda$ 
for each leaf $\lambda \in \mathcal L(T)$, as $T_\lambda$ consists of only one edge. For $v\in\mathscr{N}\setminus\{\rho\}$, the image of
$\Psi_{T_v}$ consists of (see Proposition \ref{prop:spaced-tree-parametrization} and \eqref{eq:psi_map_explicit})
\begin{equation}\label{eq:matrix_space_T_v}
    M^{T_v} := \sum_{j_v=1}^{n_v} \left(\bigotimes_{u \in \mathrm{ch}(v)}  M^{h_u}_{j_v} \right)\otimes M^{v}_{j_v}\in \left( \bigotimes_{u \in \mathrm{ch}(v)} E_u \right)\otimes E_v^* = \mathcal A_v,
\end{equation}
where $M^\bullet_{j_v}$ is the $j_v$-th column of $M^\bullet$. At the root,
\begin{equation*}
    M^{T_\rho}
:=
\sum_{j_\rho=1}^{n_\rho}
\bigotimes_{u\in\operatorname{ch}(\rho)}M_{j_\rho}^{h_u}
\in
\bigotimes_{u\in\operatorname{ch}(\rho)}E_u
=
\mathcal A_\rho.
\end{equation*}

Since $n_v
=
\dim(\mathcal A_v)
=
(\prod_{u\in\operatorname{ch}(v)} r_u)r_v$,
a product basis of $\mathcal A_v$ consists of exactly $n_v$ pure
tensors and every element of $\mathcal A_v$ is a linear
combination of at most $n_v$ pure tensors. Choosing the factors in
such a basis expansion as the columns of the matrices $M^{h_u}$
and $M^v$, and absorbing each scalar coefficient into one of these
columns, realizes the given tensor through
\eqref{eq:matrix_space_T_v}. Thus, $\operatorname{im}(\Psi_{T_v})=\mathcal A_v$.
The root case is analogous, with the factor $E_v^*$ omitted.
Consequently, each map $\Psi_{T_v}$ is surjective onto
$\mathcal A_v$, and hence 
the product map
$\Theta
=
\prod_{v\in\mathscr T}\Psi_{T_v}$
is surjective onto $\mathcal{P}_{\mathscr{T}, \mathbf{V}, \mathbf{r}}$.

The image of $\Psi_T$ for $A \in \mathrm{Rep}(T)$ can be computed using
Proposition \ref{prop:spaced-tree-parametrization}. Applied to our spaced tree $T$, the sum in \eqref{eq:spaced-tree-param} is over all internal nodes in $T$ (both the internal nodes in $\mathscr{T}$ and the rank nodes) and 
the matrices appearing in the formula are $A^\lambda$ where $\lambda$ is a leaf, $M^v$ where $v$ is 
an internal node of $T$ that is not a rank node, and $M^{h_u}$ where $u$ is a child of such an internal node $v$:
\begin{align}\label{eq:proof_eq_correspondence}
    (\Psi_T(A))_{(j_\lambda)_{\lambda \in \mathcal{L}}} 
    &= \sum_{(\alpha_v)_{v \in \mathscr{T}\setminus \{\rho\}}} 
\sum_{(j_v)_{v \in \mathscr{N}}}
\left( \prod_{u \in \mathrm{ch}(\rho)} M^{h_u}_{\alpha_u,j_\rho}
\prod_{v\in\mathscr{N}\setminus\{\rho\}}
\left( \prod_{u \in \mathrm{ch}(v)} M^{h_u}_{\alpha_u, j_v} 
\right)
M^v_{\alpha_v, j_v}
\right)
\prod_{\lambda\in\mathscr L}    A^\lambda_{j_\lambda,\alpha_\lambda} \nonumber\\
&= \sum_{(\alpha_v)_{v \in \mathscr{T}\setminus \{\rho\}}} 
A^\rho_{(\alpha_u)_{u\in\operatorname{ch}(\rho)}}
    \prod_{v\in\mathscr N\setminus\{\rho\}}
    A^v_{(\alpha_u)_{u\in\operatorname{ch}(v)},\alpha_v}
    \prod_{\lambda\in\mathscr L}
    A^\lambda_{j_\lambda,\alpha_\lambda}.
\end{align}
Here we again used $(\alpha_v)_{v \in \mathscr{T}\setminus \{\rho\}} \in \prod_{v \in \mathscr{T}\setminus \{\rho\}}[n_{h_{v}}]$ and $(j_v)_{v \in \mathscr{N}} \in \prod_{v \in \mathscr{N}}[n_v]$.
The second equality follows by applying $\Theta$ to the representation $A \in \mathrm{Rep}(T)$ which gives
\begin{equation*}
    \sum_{j_\rho=1}^{n_\rho} \prod_{u \in \mathrm{ch}(\rho)} M^{h_u}_{\alpha_u, j_\rho} =A^\rho_{(\alpha_u)_{u \in \mathrm{ch}(\rho)}} \quad \mbox{  and  } \quad
\sum_{j_v=1}^{n_v} 
\left(
\prod_{u \in \mathrm{ch}(v)} M^{h_u}_{\alpha_u, j_v} \right) \cdot M^v_{\alpha_v,j_v} = A^v_{(\alpha_u)_{u \in \mathrm{ch}(v)}, \alpha_v}
\end{equation*}
for $\rho$ and  $v \in \mathscr{N}\setminus \{\rho\}$, respectively. We recognize \eqref{eq:TTN_param} in \eqref{eq:proof_eq_correspondence} and together with the surjectivity of $\Theta$, we have therefore shown that 
$\mathrm{im}(\Psi_T) = \mathrm{im}(\widehat{\Phi}_{\mathscr{T}, \mathbf{V}, \mathbf{r}} \circ \Theta) 
    = \mathrm{im}(\widehat{\Phi}_{\mathscr{T}, \mathbf{V}, \mathbf{r}})$.
We can conclude that
\begin{equation*}
    \widehat{\mathcal{V}}_T = \overline{\mathrm{im}(\Psi_T)} = \overline{\mathrm{im}(\widehat{\Phi}_{\mathscr{T}, \mathbf{V}, \mathbf{r}})} = \widehat{\mathrm{TTN}}_{\mathscr{T}, \mathbf{V}, \mathbf{r}}.
\end{equation*}
This shows that as projective varieties $\mathrm{TTN}_{\mathscr{T}, \mathbf{V}, \mathbf{r}} = \V_T$.
\end{proof}

\begin{example} We present an extended example to illustrate the proof of Theorem \ref{thm:TTN=Spaced_tree}. For this we use the tree tensor network from Example \ref{ex:main_TTN} based on Figure \ref{fig:TTN_combinatorial_tree} which we replicate in Figure \ref{fig:tree_tensor_network}. Again we show that the image of $\Psi_T$ for the corresponding spaced tree $T$ coincides with the image of the parametrization $\widehat{\Phi}_{\mathscr T,\mathbf V,\mathbf r}$ for a corresponding choice of vector spaces.
Let us first recall the TTN parametrization for the tree defined in Figure \ref{fig:TTN_vs_Markov_tree_model} (see \eqref{eq:example_TTN_N_4})
\begin{equation}\label{eq:TTN_param_ex}
    (\widehat{\Phi}_{\mathscr T,\mathbf V,\mathbf r}(\mathbf{A}))_{j_{\lambda_1},j_{\lambda_2},j_{\lambda_3},j_{\lambda_4}} = \sum_{\substack{\beta_1, \beta_2 \\ \alpha_1, \alpha_2, \alpha_3, \alpha_4}} A^\rho_{\beta_1,\beta_2}A^{v_1}_{\alpha_1,\alpha_2, \beta_1}A^{v_2}_{\alpha_3,\alpha_4, \beta_2} A^{\lambda_1}_{j_{\lambda_1},\alpha_1} A^{\lambda_2}_{j_{\lambda_2},\alpha_2}A^{\lambda_3}_{j_{\lambda_3},\alpha_3} A^{\lambda_4}_{j_{\lambda_4},\alpha_4},
\end{equation}
for $\mathbf{A} \in \mathcal{P}_{\mathscr{T}, \mathbf{V}, \mathbf{r}}$ with $A^v \in \mathcal{A}_v$ for $v \in \mathscr{N}$ and $A^\lambda \in \mathcal{A}_\lambda$, and $\beta_i \in [r_{v_i}]$, $\alpha_i \in [r_{\lambda_i}]$. We again use the same notational simplification for the indices as in Example \ref{ex:main_TTN}.
Our goal is to realize this parametrization via a spaced tree $T$. 
The data attached to $T$ consists of the description of $T$ as a combinatorial tree together with the various vector spaces attached to each vertex in $T$. The combinatorial tree is obtained from $\mathscr T$ using the procedure outlined in the proof of Theorem \ref{thm:TTN=Spaced_tree}. The result is depicted in Figure \ref{fig:Markov_spaced_tree}. 
Next, we equip the nodes of this tree with vector spaces.  In Figure \ref{fig:TTN_spaced_tree_w_spaces}, we present the same spaced tree as in Figure \ref{fig:Markov_spaced_tree}, only with the nodes exchanged for the associated vector spaces. 
\begin{figure}[htbp]
\centering
\setlength{\abovecaptionskip}{3pt}
\setlength{\belowcaptionskip}{-3pt}

\begin{subfigure}[t]{0.48\textwidth}
\centering
\begin{tikzpicture}[
  scale=0.85,
  transform shape,
  x=1.0cm, y=1.0cm,
  every node/.style={font=\small},
  leaf/.style={circle, fill, inner sep=1.6pt},
  internal/.style={circle, draw, inner sep=1.6pt},
  rank/.style={circle, draw, fill=gray!50, inner sep=1.6pt},
  edge/.style={line width=0.6pt}
]
\node[internal, label=above:$V_{\rho}$] (q) at (0,0) {};

\node[rank, label=left:$E_{v_1}$] (h11) at (-1,-1) {};
\node[rank, label=right:$E_{v_2}$] (h12) at (1,-1) {};

\node[internal, label=left:$V_{v_1}$] (q11) at (-2,-2) {};
\node[internal, label=right:$V_{v_2}$] (q12) at (2,-2) {};

\node[rank, label=left:$E_{\lambda_1}$] (h21) at (-3,-3) {};
\node[rank, label=right:$E_{\lambda_2}$] (h22) at (-1,-3) {};
\node[rank, label=left:$E_{\lambda_3}$] (h23) at (1,-3) {};
\node[rank, label=right:$E_{\lambda_4}$] (h24) at (3,-3) {};

\node[leaf, label=below:$V_{\lambda_1}$] (l1) at (-3,-4) {};
\node[leaf, label=below:$V_{\lambda_2}$] (l2) at (-1,-4) {};
\node[leaf, label=below:$V_{\lambda_3}$] (l3) at (1,-4) {};
\node[leaf, label=below:$V_{\lambda_4}$] (l4) at (3,-4) {};

\draw[edge, green, line width=1pt] (q) -- (h11);
\draw[edge, green, line width=1pt] (q) -- (h12);
\draw[edge, blue, line width=1pt] (q11) -- (h11);
\draw[edge] (q12) -- (h12);
\draw[edge, blue, line width=1pt] (q11) -- (h21);
\draw[edge, blue, line width=1pt] (q11) -- (h22);
\draw[edge] (q12) -- (h23);
\draw[edge] (q12) -- (h24);
\draw[edge, red, line width=1pt] (l1) -- (h21);
\draw[edge] (l2) -- (h22);
\draw[edge] (l3) -- (h23);
\draw[edge] (l4) -- (h24);

\end{tikzpicture}
\caption{Markov spaced tree with associated vector spaces.}
\label{fig:TTN_spaced_tree_w_spaces}
\end{subfigure}

\vspace{1em}

\begin{subfigure}[t]{0.30\textwidth}
\centering
\begin{tikzpicture}[
  scale=0.9,
  transform shape,
  x=1.0cm, y=1.0cm,
  every node/.style={font=\small},
  leaf/.style={circle, fill, inner sep=1.6pt},
  internal/.style={circle, draw, inner sep=1.6pt},
  rank/.style={circle, draw, fill=gray!50, inner sep=1.6pt},
  edge/.style={line width=0.6pt}
]
\node[rank, label=below:$E_{\lambda_1}$] (h21) at (0,0) {};
\node[leaf, label=below:$V_{\lambda_1}$] (l1) at (-1,0) {};
\draw[edge, red, line width=1pt] (h21) -- (l1);
\end{tikzpicture}
\caption{Subtree $T_{\lambda_1}$.}
\label{fig:Tl1}
\end{subfigure}
\hfill
\begin{subfigure}[t]{0.30\textwidth}
\centering
\begin{tikzpicture}[
  scale=0.9,
  transform shape,
  x=1.0cm, y=1.0cm,
  every node/.style={font=\small},
  leaf/.style={circle, fill, inner sep=1.6pt},
  internal/.style={circle, draw, inner sep=1.6pt},
  rank/.style={circle, draw, fill=gray!50, inner sep=1.6pt},
  edge/.style={line width=0.6pt}
]
\node[rank, label=below:$E_{v_1}$] (h11) at (-1,0) {};
\node[rank, label=below:$E_{v_2}$] (h12) at (1,0) {};
\node[internal, label=below:$V_{\rho}$] (q0) at (0,0) {};

\draw[edge, green, line width=1pt] (h11) -- (q0);
\draw[edge, green, line width=1pt] (h12) -- (q0);
\end{tikzpicture}
\caption{Subtree $T_{\rho}$.}
\label{fig:Tq0}
\end{subfigure}
\hfill
\begin{subfigure}[t]{0.30\textwidth}
\centering
\begin{tikzpicture}[
  scale=0.9,
  transform shape,
  x=1.0cm, y=1.0cm,
  every node/.style={font=\small},
  leaf/.style={circle, fill, inner sep=1.6pt},
  internal/.style={circle, draw, inner sep=1.6pt},
  rank/.style={circle, draw, fill=gray!50, inner sep=1.6pt},
  edge/.style={line width=0.6pt}
]
\node[rank, label=above:$E_{\lambda_1}$] (h21) at (-1,0) {};
\node[rank, label=right:$E_{\lambda_2}$] (h22) at (0,-1) {};
\node[rank, label=above:$E_{v_1}$] (h11) at (1,0) {};
\node[internal, label=above:$V_{v_1}$] (q11) at (0,0) {};

\draw[edge, blue, line width=1pt] (h21) -- (q11);
\draw[edge, blue, line width=1pt] (h22) -- (q11);
\draw[edge, blue, line width=1pt] (h11) -- (q11);
\end{tikzpicture}
\caption{Subtree $T_{v_1}$.}
\label{fig:T_q11}
\end{subfigure}

\vspace{0.5em}

\caption{The spaced tree associated to the TTN in Figure \ref{fig:TTN_combinatorial_tree}, together with its associated vector spaces and the highlighted subtrees $T_{\lambda_1}$, $T_{\rho}$, and $T_{v_1}$.}
\label{fig:TTN_spaced_tree_combined}
\end{figure}
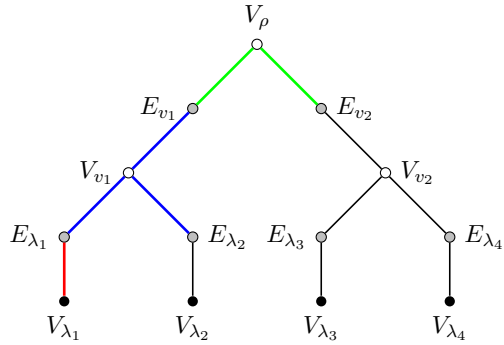
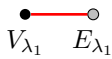
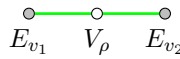
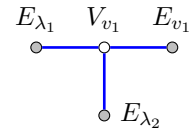
\\
We can now deconstruct $T$ into subtrees. Each subtree will give a representation space $\mathrm{Rep}(T_v)$ for $v \in \{\rho, v_1, v_2, \lambda_1, \lambda_2, \lambda_3, \lambda_4\}$. 
For the terminal leaves $\lambda_i$ all subtrees $T_{\lambda_i}$ have the same structure. As an example, we highlight the subtree $T_{\lambda_1}$ in Figure \ref{fig:Tl1}. Since these subtrees only include one edge $\{\lambda_i, h_{\lambda_i}\}$, they are represented via $\mathrm{Rep}(T_{\lambda_i}) = \{A^{\lambda_i}: A^{\lambda_i} \in \mathrm{Mat}_{n_{\lambda_i} \times n_{h_{\lambda_i}}}\}$ and 
\begin{equation*}
    \Psi_{T_{\lambda_i}}(A^{\lambda_i}) = A^{\lambda_i} = (A^{\lambda_i}_{j_{\lambda_i},\alpha_i}) \in V_{\lambda_i} \otimes E_{\lambda_i}^* = \mathcal{A}_{\lambda_i}.
\end{equation*}
In the next layer, we consider the subtree $T_{v_1}$ together with the associated matrices on the three edges: $M^{h_{\lambda_1}}\in\mathrm{Mat}_{n_{h_{\lambda_1}} \times n_{v_1}}$, $M^{h_{\lambda_2}}\in\mathrm{Mat}_{n_{h_{\lambda_2}} \times n_{v_1}}$, $M^{v_1}\in\mathrm{Mat}_{n_{h_{v_1}} \times n_{v_1}}$. The representation map for this star tree is 
\begin{equation*}
    \Psi_{T_{v_1}}(M^{h_{\lambda_1}}, M^{h_{\lambda_2}}, M^{v_1}) = M^{T_{v_1}} = (M^{T_{v_1}}_{\alpha_1, \alpha_2, \beta_1})\in E_{\lambda_1} \otimes E_{\lambda_2} \otimes E_{v_1}^* = \mathcal{A}_{v_1},
\end{equation*}
with $M^{T_{v_1}}_{\alpha_1, \alpha_2,\beta_1} = \sum_{j_{v_1}=1}^{n_{v_1}}M^{h_{\lambda_1}}_{\alpha_1, j_{v_1}}M^{h_{\lambda_2}}_{\alpha_2, j_{v_1}}M^{v_1}_{\beta_1,j_{v_1}}$. Note that this corresponds exactly to the construction in \eqref{eq:matrix_space_T_v}. Similarly, the representation map of the star tree $T_{v_2}$ centered around node $v_2$ is 
\begin{equation*}
    \Psi_{T_{v_2}}(M^{h_{\lambda_3}}, M^{h_{\lambda_4}}, M^{v_2}) = M^{T_{v_2}} = (M^{T_{v_2}}_{\alpha_3, \alpha_4, \beta_2})\in E_{\lambda_3} \otimes E_{\lambda_4} \otimes E_{v_2}^* = \mathcal{A}_{v_2},
\end{equation*}
Finally, we consider the star tree centered around $\rho$. This gives the representation map
\begin{equation*}
    \Psi_{T_{\rho}}(M^{h_{v_1}}, M^{h_{v_2}}) = M^{T_{\rho}} = (M^{T_{\rho}}_{\beta_1, \beta_2})\in E_{v_1} \otimes E_{v_2} = \mathcal{A}_\rho.
\end{equation*}
With this we have the complete set of maps for the subtrees:
\begin{equation*}
    \{\Psi_{T_{\rho}},\Psi_{T_{v_1}}, \Psi_{T_{v_2}}, \Psi_{T_{\lambda_1}}, \Psi_{T_{\lambda_2}}, \Psi_{T_{\lambda_3}}, \Psi_{T_{\lambda_4}}\}.
\end{equation*}
Given the set of subtrees we can now construct the full map $\Psi_T$ by applying Proposition \ref{prop:spaced-tree-parametrization}:
\begin{align*}
    &(\Psi_T(A))_{j_{\lambda_1},j_{\lambda_2},j_{\lambda_3},j_{\lambda_4}} = \sum_{(j_v)_{v\in T \setminus \mathcal{L}(T)}} \prod_{\{u,w\} \in \mathcal{E}(T)} (A_{uw})_{j_u j_w} \\
    &= \sum_{\substack{j_{v_1}, j_{v_2}, j_\rho, \\ j_{h_{\lambda_1}}, j_{h_{\lambda_2}}, j_{h_{\lambda_3}}, j_{h_{\lambda_4}}, \\ j_{h_{v_1}}, j_{h_{v_2}}}}
    M^{h_{v_1}}_{j_\rho, j_{h_{v_1}}} M^{h_{v_2}}_{j_\rho, j_{h_{v_2}}}M^{v_1}_{j_{h_{v_1}}, j_{v_1}}M^{v_2}_{j_{h_{v_2}}, j_{v_2}}M^{h_{\lambda_1}}_{j_{v_1}, j_{h_{\lambda_1}}}M^{h_{\lambda_2}}_{j_{v_1}, j_{h_{\lambda_2}}}M^{h_{\lambda_3}}_{j_{v_2}, j_{h_{\lambda_3}}}M^{h_{\lambda_4}}_{j_{v_2}, j_{h_{\lambda_4}}} \\ & \qquad\qquad\qquad\qquad\qquad \times   A^{\lambda_1}_{j_{\lambda_1},j_{h_{\lambda_1}}}A^{\lambda_2}_{j_{\lambda_2},j_{h_{\lambda_2}}}A^{\lambda_3}_{j_{\lambda_3},j_{h_{\lambda_3}}}A^{\lambda_4}_{j_{\lambda_4},j_{h_{\lambda_4}}}\\
\end{align*}
Using the map $\Theta$, more specifically the maps $\Psi_{T_{\rho}},\Psi_{T_{v_1}}$ and $\Psi_{T_{v_2}}$, we obtain elements in the parameter spaces $\mathcal{A}_v,\,\mathcal{A}_\rho$ of $\mathcal{P}_{\mathscr{T}, \mathbf{V}, \mathbf{r}}$ from elements of $\mathrm{Rep}(T)$
\begin{align*}
    M^{T_{v_1}}_{j_{h_{v_1}}, j_{h_{\lambda_1}}, j_{h_{\lambda_2}}} &= \sum_{j_{v_1}=1}^{n_{v_1}}M^{v_1}_{j_{h_{v_1}}, j_{v_1}}M^{h_{\lambda_1}}_{j_{v_1}, j_{h_{\lambda_1}}}M^{h_{\lambda_2}}_{j_{v_1}, j_{h_{\lambda_2}}}, \\
    M^{T_{v_2}}_{j_{h_{v_2}}, j_{h_{\lambda_3}}, j_{h_{\lambda_4}}} &= \sum_{j_{v_2}=1}^{n_{v_2}}M^{v_2}_{j_{h_{v_2}}, j_{v_2}}M^{h_{\lambda_3}}_{j_{v_2}, j_{h_{\lambda_3}}}M^{h_{\lambda_4}}_{j_{v_2}, j_{h_{\lambda_4}}}, \\
    M^\rho_{j_{h_{v_1}}, j_{h_{v_2}}} &= \sum_{j_\rho=1}^{n_\rho}M^{h_{v_2}}_{j_\rho, j_{h_{v_2}}}M^{h_{v_1}}_{j_\rho, j_{h_{v_1}}},
\end{align*}
and rename $j_{h_{\lambda_i}} = \alpha_i$, $j_{h_{v_i}} = \beta_i$. Translating the spaced-tree vertex dimensions into TTN ranks, that is, $r_{v_i} = n_{h_{v_i}}$ and $r_{\lambda_i} = n_{h_{\lambda_i}}$, we can rewrite the expression as
\begin{align*}
    (\Psi_T(A))_{j_{\lambda_1},j_{\lambda_2},j_{\lambda_3},j_{\lambda_4}} = \sum_{\substack{\beta_1, \beta_2 \\ \alpha_1, \alpha_2, \alpha_3, \alpha_4}} M^\rho_{\beta_1, \beta_2}M^{T_{v_1}}_{\beta_1, \alpha_1, \alpha_2}M^{T_{v_2}}_{\beta_2, \alpha_3, \alpha_4}A^{\lambda_1}_{j_{\lambda_1},\alpha_1} A^{\lambda_2}_{j_{\lambda_2},\alpha_2}A^{\lambda_3}_{j_{\lambda_3},\alpha_3} A^{\lambda_4}_{j_{\lambda_4},\alpha_4}.
\end{align*}
Identifying $M^{T_{\rho}} = A^\rho$, $M^{T_{v_1}} = A^{v_1}$ and $M^{T_{v_2}} = A^{v_2}$ from the parametrization in \eqref{eq:TTN_param_ex}, we see that the image of $\widehat{\Phi}_{\mathscr T,\mathbf V,\mathbf r}$ is equal to the image of $\Psi_T$. 
\end{example}

\section{Equations of tree tensor network varieties}  \label{sec:equations}
Our aim in this section is to prove that 
the equations defining the tree tensor network variety $\mathrm{TTN}_{\mathscr{T}, \mathbf{V}, \mathbf{r}}$ set-theoretically actually generate the prime ideal $\mathcal{I}(\mathrm{TTN}_{\mathscr{T}, \mathbf{V}, \mathbf{r}})$.  Recall from Proposition \ref{prop:TTN_set_theoretic} that $\mathrm{TTN}_{\mathscr{T}, \mathbf{V}, \mathbf{r}}$ is set-theoretically cut out by the polynomials
\begin{equation} \label{eq:ideal-generators}
     \bigcup_{v \in \mathscr{T}\setminus \{\rho\}} \{(r_v+1)\textrm{-minors of  } \psi^{(v)}\},
\end{equation}
where $\psi^{(v)}$ is a flattening \eqref{eq:flattening} of the 
tensor $\psi \in \bigotimes_{\lambda \in \mathscr{L}} V_\lambda = W$.
In coordinates, 
$\psi^{(v)}$ is a 
$(\prod_{\lambda \in \mathscr{L}(v)} \dim(V_\lambda)) \times 
(\prod_{\lambda \in\mathscr{L}\setminus \mathscr{L}(v)}\dim(V_\lambda))$ matrix where the 
entry in the row indexed by $(j_\lambda \, : \, \lambda \in \mathscr{L}(v))$ and the column indexed
by $(j_\lambda \, : \, \lambda \in \mathscr{L} \setminus \mathscr{L}(v))$ is 
equal to $\psi_{(j_\lambda)_{\lambda \in \mathscr{L}}}$. 

We prove the statement by using Theorem \ref{thm:TTN=Spaced_tree} which states that as projective varieties $\mathrm{TTN}_{\mathscr{T},\mathbf{V}, \mathbf{r}} = \mathcal{V}_T$. Therefore, it suffices to show that the homogeneous prime ideal of the spaced tree $T$ corresponding to $\mathrm{TTN}_{\mathscr{T},\mathbf{V}, \mathbf{r}}$ is generated by the above polynomials. This transition from a TTN parametrization to the general Markov model on a spaced tree allows us to apply Theorem \ref{thm:ideal}. 
\\
Before stating the result, we make some observations that will be necessary. Recall Definition \ref{def:contracted_star_tree} of the contracted star tree $\mathfrak{s}_q(T)$ for 
$q \in \mathrm{int}(T)$. The tree $\mathfrak{s}_q(T)$ comes with a set of associated leaf classes $\mathcal C(q)$, each leaf of $\mathfrak{s}_q(T)$ is labeled by one of these classes, and $D_q = |\mathcal C(q)|$ is the degree of the vertex $q$. Let $\mathcal{V}_{\mathfrak{s}_q(T)}$ be the associated spaced-tree variety (see Remark \ref{rmk:s_q_variety}). We distinguish two kinds of vertices in the following. First, we consider the case when $q=h_v$ is a rank node with $D_{h_v}=2$ and $\mathrm{dim}(V_{h_v}) = n_{h_v} =: r_v$, where $v$ is the unique child of $h_v$. For a rank node, the star tree $\mathfrak{s}_{h_v}(T)$ introduces leaf classes $\mathcal C(h_v) = \{C, \overline{C}\}$, where $C$ is an element in $\mathcal C(v)$ and $\overline{C}$ is the union of the remaining classes of $v$. The associated variety $\V_{\mathfrak{s}_{h_v}(T)}$ is
\begin{align*}
    \V_{\mathfrak{s}_{h_v}(T)} &= \sigma_{r_v}\left(\mathrm{Seg}\left( \PP\left(K_C\right)\times \PP\left(K_{\overline{C}}\right)\right)\right) \subseteq \PP\left( K_T \right) \\
    &=\{[M] \in \PP\left(K_C  \otimes K_{\overline{C}}\right) : \mathrm{rank}(M) \leq r_v \},
\end{align*}
that is, it is the determinantal variety of matrices with rank at most $r_v$. Again, we denoted $K_C =\bigotimes_{v\in C} V_v$. These varieties are cut out by the homogeneous prime ideals generated by the minors of size $(r_v +1)$ \cite{bruns1988}, 
\begin{equation*}
    I_v :=\mathcal I(\V_{\mathfrak{s}_{h_v}(T)})
= \Big\langle (r_v + 1)\textrm{-minors of }\psi^{(v)} \Big\rangle.
\end{equation*}
These ideals are generated by precisely the polynomials that appear in 
 \eqref{eq:ideal-generators}.
Second, we consider the case of non-rank non-leaf vertices $q = v \in \mathrm{int}(T)$ with $\mathrm{dim}(V_v) = n_v$. The star tree $\mathfrak{s}_v(T)$ introduces leaf
classes $\mathcal C(v) = \{C_1, ..., C_{D_v}\}$ with $D_v \geq 2$. The variety is given as the $n_v$-th secant variety of the Segre variety of $D_v$ projective spaces,
\begin{equation} \label{eq:star-tree-secant}
    \V_{\mathfrak{s}_v(T)} = \sigma_{n_v}\left(\mathrm{Seg}\left(\bigtimes_{C \in \mathcal{C}(v)} \PP\left(K_C\right)\right)\right) \subseteq \PP\left( K_T \right),
\end{equation}
or, equivalently, as the projective variety of order-$D_v$ tensors of border rank at most $n_v$. It contains order-$D_v$ tensors of tensor rank at most $n_v$ as a Zariski dense set.  
Since $\V_{\mathfrak{s}_v(T)}$ is in general not a determinantal variety, there is no uniform determinantal description of this ideal in general. In this case we simply denote the ideal as $\mathcal I(\V_{\mathfrak{s}_v(T)})$. 
At this point we know that the ideal for the general Markov model on spaced tree $T$ is (Theorem \ref{thm:ideal})
\begin{equation*}
    \mathcal{I}(\mathcal{V}_T) = \sum_{q \in \mathrm{int}(T)}\mathcal{I}(\V_{\mathfrak{s}_q(T)}),
\end{equation*}
where the sum includes both the rank nodes $h$ and the non-rank internal nodes $v$. 
Comparing this with \eqref{eq:ideal-generators}, we need to show that the unknown ideals $\mathcal I(\V_{\mathfrak{s}_v(T)})$ are irrelevant in the above sum, and, in fact, already included in the ideals $I_v$.
To do this we will first prove two lemmas and a corollary.
\begin{lemma} \label{lem:bound-tensor-rank}
Let $S \in \bigotimes_{i=1}^n \CC^{d_i}$ be a tensor where $n\geq 2$, and let $S^{(i)}$ be the flattening corresponding to the partition $\{i\} \mid \{1, \ldots, \hat{i}, \ldots, n\}$ for $i=1,\ldots, n$. If $\rank(S^{(i)}) \leq r_i$ with $r_i \leq d_i$ for $i=1,\ldots,n$, then $\mathrm{rank}(S) \leq \prod_{i=1}^{n} r_i$.
\end{lemma}
\begin{proof}
For all $i=1, \ldots, n$, there exists a subspace $U_i \subseteq \CC^{d_i}$ with $\dim(U_i) \leq r_i$ such that
\begin{equation*}
    S \in U_i \otimes \bigotimes_{j\neq i} \CC^{d_j}.
\end{equation*}
Hence $S \in \bigotimes_{i=1}^n U_i$. Now, let $\{u^i_1, ..., u^i_{r_i}\}$ be a zero-padded spanning family for $U_i$. Then we can write 
\begin{equation*}
    S = \sum_{j_1=1}^{r_1}\cdots \sum_{j_n=1}^{r_n}\alpha_{j_1,\ldots, j_n} u^1_{j_1} \otimes 
    \cdots \otimes u^{n}_{j_n} \,\, \in \bigotimes_{i = 1}^{n} U_i.
\end{equation*}
This means that the tensor rank of $S$ is at most $\prod_{i=1}^n r_i$.
\end{proof}
\begin{corollary} \label{cor:variety-inclusion}
Let $v \in \mathrm{int}(T)$ 
be a non-rank vertex
in the spaced tree $T$  with $D_v\geq 2$, 
and let 
$\{h_1, ..., h_{D_v}\}$ be the adjacent rank nodes. Then
\begin{equation} \label{eq:star-tree-containment} 
\bigcap_{i=1}^{D_v} \V_{\mathfrak{s}_{h_i}(T)} \subseteq \V_{\mathfrak{s}_v(T)} \quad \text{and} \quad \mathcal{I}(\V_{\mathfrak{s}_v(T)}) \subseteq  = 
     \sqrt{\sum_{i=1}^{D_v} \mathcal{I}(\V_{\mathfrak{s}_{h_i}(T)})}.
\end{equation}  
\end{corollary}
\begin{proof}
The leaves of the star tree $\mathfrak{s}_v(T)$ are 
labeled by classes $C_1, \ldots, C_{D_v}$. 
 Consider a tensor 
 $S \in 
 \PP(\bigotimes_{i=1}^{D_v} K_{C_i}) $ such that $S \in \V_{\mathfrak{s}_{h_i}}(T)$ for all $i=1,\ldots, D_v$.
 In other words, the flattening $S^{(i)}$ corresponding to the partition $\{i\} \mid \{1, \ldots, \hat{i}, \ldots, D_v\}$ has rank at most $r_i = r_{h_i}$ for all $i=1, \ldots, D_v$. By Lemma \ref{lem:bound-tensor-rank}, 
 the tensor rank of $S$ is bounded by $n_v := \prod_{i=1}^{D_v} r_i$ and 
 hence 
 \begin{equation*}
    S \in \V_{\mathfrak{s}_v(T)}=\sigma_{n_v}\left(\mathrm{Seg}\left(\bigtimes_{i=1}^{D_v} \PP\left(K_{C_i}\right)\right)\right). 
\end{equation*}
The second statement about the ideal inclusion follows from Hilbert's Nullstellensatz.
\end{proof}

The next lemma can be viewed as a statement about the primeness of the natural determinantal ideal which defines a subspace variety  set-theoretically (see Corollary \ref{cor: S flattening}). This result has been known in the literature (see \cite{landsberg2007, oeding2016}). We give a novel proof using the language of Markov models on spaced trees.
 \begin{lemma} \label{lem:sum-of-determinantal}
  Let $S = (S)_{j_1,\ldots, j_n}$ be a $d_1 \times \cdots \times d_n$-tensor consisting of variables and let $I_i = \langle (r_i+1)\textrm{-minors of  } S^{(i)}\rangle$ for $i=1,\ldots, n$ and $r_i \in \NN$, where $S^{(i)}$ is the flattening of $S$
  corresponding to the partition $\{i\} \mid \{1, \ldots, \hat{i}, \ldots, n\}$ and $r_i \leq d_i$ for $i=1,\ldots,n$. Then $\sum_{i=1}^n I_i$ is a prime ideal.
 \end{lemma}
 \begin{proof}
 We first construct a spaced tree. Let $T$ be the tree with an internal node $q$ connected to rank nodes $h_1, \ldots, h_n$, and let $\lambda_1, \ldots, \lambda_n$ be leaves of $T$ where $\lambda_i$ is connected to $h_i$. Attach vector spaces $\CC^{d_i}$ to the leaves $\lambda_i$ and $\CC^{r_i}$ to the rank nodes $h_i$ for $i=1,\ldots,n$. Furthermore, attach $\CC^m$ to the node $q$, that is, $n_q = m$. By Theorem \ref{thm:ideal}, the ideal of 
 the general Markov model on this spaced tree is
 $$ \mathcal{I}(\mathcal{V}_T) =  \mathcal{I}(\mathcal{V}_{\mathfrak{s}_q(T)}) + \sum_{i=1}^n \mathcal{I}(\mathcal{V}_{\mathfrak{s}_{h_i}(T)}) = 
 \mathcal{I}(\mathcal{V}_{\mathfrak{s}_q(T)}) + \sum_{i=1}^n I_i.$$
 Note that the above is true for any choice of $m$. 
 At the same time $\mathcal{V}_{\mathfrak{s}_q(T)}$
 contains all tensors in $\PP(\bigotimes_{i=1}^n \CC^{d_i})$ of tensor rank at most $m$. Therefore, if we choose $m \geq \prod_{i=1}^{n}d_i \geq \prod_{i=1}^n r_i$, this secant variety will be equal to the entire tensor space $\PP(\bigotimes_{i=1}^n \CC^{d_i})$. For this choice, 
 $\mathcal{I}(\mathcal{V}_{\mathfrak{s}_q(T)})=0$ and  $\mathcal{I}(\mathcal{V}_T) = \sum_{i=1}^n I_i$. We emphasize that $I_i$ are the same for any choice of $m$. Since $\mathcal{V}_T$ is irreducible, its ideal $\mathcal{I}(\mathcal{V}_T)$ is prime and so is $\sum_{i=1}^n I_i$. This proves the result.
 \end{proof}

Finally we are ready to prove our main theorem in this section.
\main*

\begin{proof}
Using Corollary \ref{cor:variety-inclusion}, we know that 
\begin{equation*}
     \mathcal{I}(\V_{\mathfrak{s}_v(T)}) \subseteq \sqrt{\sum_{i=1}^{D_v} \mathcal{I}(\V_{\mathfrak{s}_{h_i}(T)})}
\end{equation*}
for every non-rank non-leaf $v$ of the spaced tree $T$
corresponding to $\mathrm{TTN}_{\mathscr{T}, \mathbf{V}, \mathbf{r}}$ and its adjacent rank nodes $\{h_1, ..., h_{D_v}\}$.
Lemma \ref{lem:sum-of-determinantal} implies that
the sum of the determinantal ideals above is prime. This means that 
\begin{equation*}
     \mathcal{I}(\V_{\mathfrak{s}_v(T)}) \subseteq \sum_{i=1}^{D_v} \mathcal{I}(\V_{\mathfrak{s}_{h_i}(T)})
\end{equation*}
for every non-rank non-leaf $v$. Now Theorem \ref{thm:ideal} gives the result. 
\end{proof}

We would like to record two immediate corollaries of this theorem in the cases of tensor train and subspace varieties via Corollary \ref{cor: TT flattening} and Corollary \ref{cor: S flattening}, respectively.
\begin{corollary} \label{cor: main-TT}
For fixed $\mathbf{d}=(d_1, \ldots, d_N)$ and $\mathbf{r}=(r_0=1, r_1, \ldots, r_{N-1}, r_N=1)$, let $\psi$ be a $d_1 \times \cdots \times d_N$ tensor of indeterminates and let $\psi^{(i)}$ be the flattening corresponding to the partition $\{1,\ldots, i\} \mid \{i+1, \ldots, N\}$. Then the homogeneous prime ideal of $\mathrm{TT}_{\mathbf{d}, \mathbf{r}}$ in $\mathbb{C}[\psi]$ is
\[
\mathcal I(\mathrm{TT}_{\mathbf{d}, \mathbf{r}})
= \sum_{i=1}^{N-1} \Big\langle (r_i + 1)\textrm{-minors of }\psi^{(i)} \Big\rangle.
\]    
\end{corollary}
\begin{corollary} \label{cor: main-subspace}
For fixed $\mathbf{d}=(d_1, \ldots, d_N)$ and $\mathbf{r} = (r_1, \ldots, r_N)$, let $\psi$ be a $d_1 \times \cdots \times d_N$ tensor of indeterminates and let $\psi^{(i)}$ be the flattening corresponding to the partition $\{i\} \mid \{1, \ldots, \widehat{i}, \ldots, N\}$.
 Then the homogeneous prime ideal of $\mathrm{S}_{\mathbf{d}, \mathbf{r}}$ in $\mathbb{C}[\psi]$ is
\[
\mathcal I(\mathrm{S}_{\mathbf{d}, \mathbf{r}})
= \sum_{i=1}^{N} \Big\langle (r_i + 1)\textrm{-minors of }\psi^{(i)} \Big\rangle.
\]    
\end{corollary}
\begin{remark}
Corollary \ref{cor: main-TT} settles the first part of Conjecture 5.10 in \cite{BFHP25}, which was based on a conjecture in unpublished notes of Bernd Sturmfels. Corollary \ref{cor: main-subspace} has been established via a different derivation in \cite{landsberg2007, oeding2016}. 
\end{remark}

\section{Gr\"obner bases of order $3$ tensor train varieties}\label{sec: GB}
It is natural to ask whether the minors generating the ideal
$\mathcal{I}(\mathrm{TTN}_{\mathscr{T}, \mathbf{V}, \mathbf{r}})$ as in Theorem \ref{thm: main} form a Gr\"obner basis of this ideal with respect to some term order. It is well known that for an $p\times q$ matrix $X$ and for $t \leq \min(p,q)$ the set of $t$-minors of $X$ is a Gr\"obner basis with respect to any diagonal term order; see \cite{N86,CGG90,M94,St90} and \cite[Chapter 4]{BCRV22}.
Since $\mathcal{I}(\mathrm{TTN}_{\mathscr{T}, \mathbf{V}, \mathbf{r}})$ is 
a sum of precisely such determinantal ideals one hopes for a similar result. 

In this last section, we focus on tensor train varieties and report what is already known in this case. As in Section~\ref{sec:TTN} we fix  positive integer vectors $\mathbf{d} = (d_1,\dots,d_N)$ and $\mathbf{r} = (r_0=1, r_1,\dots,r_{N-1}, r_N=1)$ and let $\mathrm{TT}_{\mathbf{d},\mathbf{r}}$ be the tensor train variety defined by this data. Corollary \ref{cor: main-TT} gives the equations defining $\mathcal{I}(\mathrm{TT}_{\mathbf{d}, \mathbf{r}})$. They are the union of the $(r_i+1)$-minors of the flattening $\psi^{(i)}$ for $i=1, \ldots, N-1$. Recall that $\psi^{(i)}$ is the $(d_1\cdots d_i) \times (d_{i+1} \cdots d_N) $ matrix where $(\psi^{(i)})_{(j_1 \cdots j_i),(j_{i+1} \cdots j_N)}$, the entry in the row indexed by $(j_1, \ldots, j_i) \in [d_1] \times \cdots \times [d_i]$ and in the column indexed by 
$(j_{i+1}, \ldots, j_N) \in [d_{i+1}] \times \cdots \times [d_N]$, is equal to $\psi_{j_1 \cdots j_N}$. First, we state a conjecture that is a refinement of the second part of Conjecture 5.10 in \cite{BFHP25} which states that the generating minors of $\mathcal{I}(\mathrm{TT}_{\mathbf{d}, \mathbf{r}})$ form a Gr\"obner basis. 
\begin{conjecture} \label{conj: minors-form-GB}
With respect to a diagonal term order, the polynomials
$$ \bigcup_{i=1}^{N-1} \{ (r_i +1)\text{-minors of } \psi^{(i)}\}$$
form a Gr\"obner basis of $\mathcal{I}(\mathrm{TT}_{\mathbf{k}, \mathbf{r}})$.
\end{conjecture}
 
We now describe the term order in the conjecture. For this we organize each flattening $\psi^{(i)}$ as follows.  For each choice of the indices $(j_2,\dots,j_{N-1}) \in [d_2] \times \cdots \times [d_{N-1}]$, we define 
$A_{j_2 \cdots j_{N-1}}$ to be the $d_1 \times d_N$ matrix of indeterminates 
\[  
  \bigl(A_{j_2\cdots j_{N-1}}\bigr)_{j_1 j_N} \;=\; \psi_{j_1 j_2 \cdots j_{N-1} j_N}.
\]
Inspired by \cite{illian2025grobner}, we will call $A_{j_2\cdots j_{N-1}}$ a \emph{page}. 
Then, $\psi^{(i)}$ is exactly the block matrix obtained by arranging these pages in a grid whose block-rows are indexed by $(j_2,\dots,j_{i})$ and whose block-columns are indexed by
$(j_{i+1},\dots,j_{N-1})$, both in lexicographic order.

Note that $\psi^{(1)}$ has a single block-row and
its block-columns are indexed by $(j_2,\cdots,j_{N-1})$, and $\psi^{(N-1)}$ has  a single block-column
and its block-rows are also indexed by $(j_2, \cdots, j_{N-1})$. 
\begin{example}\label{ex:concrete-flattening}
Suppose $(d_1,d_2,d_3,d_4)=(3,2,3,2)$. The six pages $A_{pq}$, $p\in\{1,2\}$, $q\in\{1,2,3\}$, are
$$A_{pq} = \begin{pmatrix} \psi_{1pq1} & \psi_{1pq2} \\ \psi_{2pq1} & \psi_{2pq2} \\ \psi_{3pq1} & \psi_{3pq2}\end{pmatrix},$$
and the flattenings of $\psi$ are as follows. First, $\psi^{(1)}$ is the single block-row
$$
\psi^{(1)} = \left(\begin{array}{cc|cc|cc|cc|cc|cc}
\psi_{1111} & \psi_{1112} & \psi_{1121} & \psi_{1122} & \psi_{1131} & \psi_{1132} & \psi_{1211} & \psi_{1212} & \psi_{1221} & \psi_{1222} & \psi_{1231} & \psi_{1232} \\
\psi_{2111} & \psi_{2112} & \psi_{2121} & \psi_{2122} & \psi_{2131} & \psi_{2132} & \psi_{2211} & \psi_{2212} & \psi_{2221} & \psi_{2222} & \psi_{2231} & \psi_{2232} \\
\psi_{3111} & \psi_{3112} & \psi_{3121} & \psi_{3122} & \psi_{3131} & \psi_{3132} & \psi_{3211} & \psi_{3212} & \psi_{3221} & \psi_{3222} & \psi_{3231} & \psi_{3232}
\end{array}\right).
$$
Next, $\psi^{(2)}$ is the $2\times3$ grid of pages
$$
\psi^{(2)} = \left(\begin{array}{cc|cc|cc}
\psi_{1111} & \psi_{1112} & \psi_{1121} & \psi_{1122} & \psi_{1131} & \psi_{1132} \\
\psi_{2111} & \psi_{2112} & \psi_{2121} & \psi_{2122} & \psi_{2131} & \psi_{2132} \\
\psi_{3111} & \psi_{3112} & \psi_{3121} & \psi_{3122} & \psi_{3131} & \psi_{3132} \\ \hline
\psi_{1211} & \psi_{1212} & \psi_{1221} & \psi_{1222} & \psi_{1231} & \psi_{1232} \\
\psi_{2211} & \psi_{2212} & \psi_{2221} & \psi_{2222} & \psi_{2231} & \psi_{2232} \\
\psi_{3211} & \psi_{3212} & \psi_{3221} & \psi_{3222} & \psi_{3231} & \psi_{3232}
\end{array}\right).
$$
Finally, $\psi^{(3)}$ is the single block-column obtained by stacking the six pages $A_{11},A_{12},A_{13},A_{21},A_{22},A_{23}$ vertically, in that order. 
\end{example}

\begin{definition}[A diagonal order]\label{defn: diagonal order}
Order the entries of the last flattening $\psi^{(N-1)}$ by reading it row by row, left to right in each row: that is, $\psi_{j_1\cdots j_N} \succ \psi_{j_1'\cdots j_N'}$ if $(j_2,\ldots,j_{N-1}) <_{\mathrm{lex}} (j_2',\ldots,j_{N-1}')$, or if these tuples are equal and $j_1<j_1'$, or if these tuples are equal and $j_1=j_1'$ and $j_N<j_N'$. 
\end{definition}
In other words, $\psi_{j_1\cdots j_N} \succ \psi_{j_1' \cdots j_N'}$ if the first variable is in an earlier page, or if both variables are in the same page and the first variable is in an earlier row in that page, or if both variables are in the same row of the same page, the first variable is in an earlier column. Since every variable $\psi_{j_1\cdots j_N}$ occurs exactly once among the entries of $\psi^{(N-1)}$, this defines a single total order $\succ$ on all the entries of $\psi$,
hence induces a lexicographic term order on the polynomial ring $\CC[\psi]$. 
\begin{proposition}\label{prop:global-order}
The monomial order $\succ$ of Definition~\ref{defn: diagonal order} is a diagonal order on 
each $\psi^{(i)}$, $i=1,\ldots, N-1$. 
 In other words, 
 for any minor in any of the flattening matrices, the leading monomial with respect to $\succ$ is equal to the product of the diagonal entries of the submatrix whose determinant is the minor.
\end{proposition}
\begin{proof}
The claim is  clear for $i=N-1$. Otherwise, identify each row of $\psi^{(i)}$ with a tuple $(j_1,\ldots,j_i)$ and each column with the tuple $(j_{i+1},\ldots,j_N)$. Now let $f$ be a $t$-minor of $\psi^{(i)}$ that is the determinant of the submatrix $M$ with rows $r_1 > \cdots >r_t$ and columns $c_1> \cdots >c_t$. 
Observe that  $r_k = (j_1, \ldots, j_i) > r_\ell = (j_1', \ldots, j_i')$ if $(j_2, \ldots, j_i) <_{\mathrm{lex}} (j_2', \ldots, j_i')$, 
or if these tuples are equal and $j_1 < j_1'$. Similarly, $c_k = (j_{i+1}, \ldots, j_N) > c_\ell = (j_{i+1}', \ldots, j_N')$ if $(j_{i+1}, \ldots, j_{N-1}) <_{\mathrm{lex}} (j_{i+1}', \ldots, j_{N-1}')$, 
or if these tuples are equal and $j_N < j_N'$.
Now we see that 
$$\psi_{r_1c_1} \succ \psi_{r_1c_2} \succ \cdots \succ \psi_{r_1c_t} \succ \psi_{r_2c_1} \succ \cdots \psi_{r_2c_t} \succ \cdots \succ \psi_{r_tc_1} \succ \cdots \succ \psi_{r_tc_t}.$$
In other words, the entries of $M$ are ordered lexicographically row by row, left to right in each row. Such an ordering picks the product of the diagonal terms of $M$ as the initial term of $f$. 
\end{proof}

While Conjecture \ref{conj: minors-form-GB} is open in general, it is true for tensor trains where $N=3$, and 
$\mathbf{d}=(d_1, d_2, d_3)$ and $\mathbf{r} = (r_1, r_2)$ are arbitrary. We refer to this case as an order $3$ tensor train. In this case, the ideal $\mathcal{I}(\mathrm{TT}_{\mathbf{d}, \mathbf{r}})$ is a \emph{double determinantal ideal}.
\begin{theorem} \label{thm: order-3-GB}
\cite[Theorem 4.1]{FK20}, \cite[Theorem 1.2]{illian2025grobner} For an order $3$ tensor train with $\mathbf{d} = (d_1,d_2,d_3)$ and $\mathbf{r} = (r_1, r_2)$ the polynomials
\begin{equation*}
    \{(r_1+1)\text{-minors of } \psi^{(1)} \} \,\, \cup \, \, 
\{(r_2+1)\text{-minors of } \psi^{(2)} \},
\end{equation*}
where $\psi^{(1)}$ is the $d_1 \times (d_2d_3)$ flattening of $\psi$ and $\psi^{(2)}$ is the $(d_1d_2) \times d_3$ flattening of $\psi$, form a Gr\"obner basis with respect to the diagonal order in Definition \ref{defn: diagonal order}. 
\end{theorem}

The proofs of Theorem 4.1 in \cite{FK20} and Theorem 1.2 in \cite{illian2025grobner} are quite different. We do not know whether either can be generalized to prove Conjecture \ref{conj: minors-form-GB}. However, we like to remark that the techniques in \cite{illian2025grobner} do not extend in a straightforward way: we know that at least one crucial intermediate result (\cite[Lemma 2.12]{illian2025grobner}) is not true in its obvious generalization. 

\subsection{Degree of an order $3$ tensor train variety}
The degree of an embedded projective variety is one of its most important invariants. However, the degree of tree tensor network varieties is, in general, not known at this moment. Even for the special cases of tensor train and subspace varieties there is no general explicit formula that expresses the degree in terms of the data, for instance,  in terms of $\mathbf{d}$ and $\mathbf{r}$ in the case of tensor train varieties. At the same time, there is a nearly explicit formula for subspace varieties (Theorem 1.1 in \cite{Breiding_2024}) where the formula depends on computing the coefficients of a particular multivariate polynomial. This is derived by computing the volume of the smooth stratum of the subspace variety. In a similar spirit, a nearly explicit formula for tensor train varieties is presented in Theorem 1 of \cite{RS26}. 

Here we present a purely combinatorial procedure to compute the degree of an order $3$ tensor train variety. Our result depends on an intermediate result that appears in the proof of \cite[Theorem 4.1]{FK20}; see also Remark 4.3 in the same paper. 
\begin{proposition} For an order $3$ tensor train with $\mathbf{d} = (d_1,d_2,d_3)$ and $\mathbf{r} = (r_1, r_2)$ let $J_\succ$ be the initial ideal 
of $\mathcal{I}(\mathrm{TT}_{\mathbf{d}, \mathbf{r}})$ with respect to a diagonal term order. Then $J_\succ$ is Cohen-Macaulay. This means that $\mathcal{I}(\mathrm{TT}_{\mathbf{d}, \mathbf{r}})$
is also Cohen-Macaulay.
\end{proposition}
Our combinatorial procedure depends on the following result. 
\begin{proposition} \label{prop:degree-of-order-3} For an order $3$ tensor train with $\mathbf{d} = (d_1,d_2,d_3)$ and $\mathbf{r} = (r_1, r_2)$ let $J_\succ$ be the initial ideal 
of $\mathcal{I}(\mathrm{TT}_{\mathbf{d}, \mathbf{r}})$ with respect to a diagonal term order. Let $I^{(1)}_\succ$ and $I^{(2)}_\succ$ be the initial ideals 
of 
$$I_1 = \langle (r_1+1)\text{-minors of } \psi^{(1)}\rangle \, \, \mbox{   and   }  \,\,I_2 = \langle (r_2+1)\text{-minors of } \psi^{(2)}\rangle$$
with respect to the same term order. 
Then the minimal primes of $J_\succ$ 
form the set
$$ \{ M+N \, : \, M \in \mathrm{minAss}(I^{(1)}_\succ), \,\, 
N \in \mathrm{minAss}(I^{(2)}_\succ) \mbox{  and } \dim(M+N) = \dim(J)  \}
$$
where $\dim(J) = r_1d_1 + r_1r_2d_2 + r_2d_3 -r_1^2-r_2^2$.
Moreover, the degree of the tensor train variety $TT_{\mathbf{d}, \mathbf{r}}$ is the cardinality of this set.  
\end{proposition}
\begin{proof}
Both $I^{(1)}_\succ$ 
and $I^{(2)}_\succ$ are squarefree monomial ideals (hence Stanley-Reisner ideals of certain simplicial complexes; see Definition \ref{defn:non-intersecting paths}) we have
$$ I^{(1)}_\succ = \bigcap_{M \in \mathrm{minAss}(I^{(1)}_\succ)} M  \quad \quad  \mbox{  and  }  \quad \quad 
I^{(2)}_\succ = \bigcap_{N \in \mathrm{minAss}(I^{(2)}_\succ)} N.
$$
We note that both $I^{(1)}_\succ$
 and $I^{(2)}_\succ$ are also Cohen-Macaulay (see \cite[Chapter 4]{BCRV22}), 
 so the above prime decompositions are unmixed.
By Theorem \ref{thm: order-3-GB}, $J_\succ = I^{(1)}_\succ + I^{(2)}_\succ$, and therefore
$$J_\succ = \bigcap_{\substack{ M \in \mathrm{minAss}(I^{(1)}_\succ) \\ N \in \mathrm{minAss}(I^{(2)}_\succ)}} M + N.$$
This follows by proving the two inclusions. The inclusion $\subseteq$ is true for any two arbitrary ideals which are intersections of two collections of other ideals. For the reverse inclusion 
it is enough to prove that any monomial $m$ contained on the right hand side is also in $J_\succ$. If $m \in M$ for every minimal prime $M$ of $I^{(1)}_\succ$, clearly $m \in J_\succ$. If $m \not \in M$ for some minimal prime of $I^{(1)}_\succ$, then it must be divisible by some variable in each minimal prime $N$ of $I^{(2)}_\succ$. This means $m \in I^{(2)}_\succ$ and therefore $m \in J_\succ$.
Since $J_\succ$ is Cohen-Macaulay, its unmixed prime decomposition is given by 
the intersection of those sums $M+N$ whose dimension is equal to  
$\dim(J_\succ) = \dim(\mathcal{I}(\mathrm{TT}_{\mathbf{d}, \mathbf{r}}))$.
This dimension is equal to $r_1d_1 + r_1r_2d_2 + r_2d_3 - r_1^2-r_2^2$ by Corollary 6.5 in \cite{BFHP25}. Now the degree of $J_\succ$ is equal to the number of its distinct minimal primes, and this is also equal to the degree of $\mathcal{I}(\mathrm{TT}_{\mathbf{d}, \mathbf{r}})$. 
 \end{proof}
Although Proposition \ref{prop:degree-of-order-3} is algebraic, in order to compute the minimal primes and the degree of $J_\succ$ we do not have to do any algebraic computations since there is a beautiful combinatorial description of the minimal primes of $I^{(1)}_\succ$ and $I^{(2)}_\succ$ which gives rise to a combinatorial description of the minimal primes of $J_\succ$. 

\begin{definition} \label{defn:non-intersecting paths}
 For $\ell=1, \ldots, r_1$, let $P_\ell$ be 
a path on the flattening matrix $\psi^{(1)}$ that connects the $\ell$th entry on the last row  to the $\ell$th entry on the last column of this matrix while the path makes only "up" or "right" steps. We call a collection of the paths $(P_1, P_2, \ldots, P_{r_1})$ non-intersecting if no pair of paths contain the same variable. The minimal primes of $I^{(1)}_\succ$ are in bijection with the set of all non-intersecting collection of paths \cite[Proposition 4.4.1]{BCRV22}. In fact, the minimal prime corresponding to a non-intersecting collection of paths $(P_1, \ldots, P_{r_1})$  is
 $M = \langle \psi_{ijk}\, : \, \psi_{ijk} \not \in \cup_{\ell=1}^{r_1} P_\ell \rangle$. Similarly, for $\ell=1,\ldots, r_2$, let $Q_\ell$ be 
a path on the flattening matrix $\psi^{(2)}$ that connects the $\ell$th entry on the last row  to the $\ell$th entry on the last column of this matrix with only "up" or "right" steps. Non-intersecting collection of paths $(Q_1,\ldots, Q_{r_2})$ correspond bijectively to minimal primes of $I^{(2)}_\succ$ where such a collection gives the minimal prime $N = \langle \psi_{ijk}\, : \, \psi_{ijk} \not \in  \cup_{\ell=1}^{r_2} Q_\ell \rangle$.
\end{definition}
The following corollary allows one to compute the degree of $\mathrm{TT}_{\mathbf{d}, \mathbf{r}}$ in the order $3$ case purely combinatorially. 
\begin{corollary}
 The minimal primes of $J_\succ$ are of the form 
 $$ \langle \psi_{ijk} \, : \, \psi_{ijk} \not \in (\cup_{\ell=1}^{r_1} P_\ell) \cap  (\cup_{\ell=1}^{r_2} Q_\ell) \mbox{  with  } |(\cup_{\ell=1}^{r_1} P_\ell) \cap  (\cup_{\ell=1}^{r_2} Q_\ell)| = \dim(J) \rangle$$
 where $\dim(J)= r_1d_1 + r_1r_2d_2 +r_2d_3-r_1^2-r_2^2$ and $(P_1, \ldots, P_{r_1})$ and $(Q_1, \ldots, Q_{r_2})$ are non-intersecting collections of paths on $\psi^{(1)}$ and $\psi^{(2)}$, respectively. 
\end{corollary}
In Table \ref{tab:tt-degrees} we compute the degrees of some  tensor train varieties of order $3$ with a {\tt Julia} code, as provided in \cite{TensorTrainDegrees}. The data for $(r_1, r_2) = (r,1)$ or $(r_1, r_2) = (1,r)$ is not displayed since there is a formula in these cases. 

\begin{proposition}
Let $\mathrm{TT}_{\mathbf{d}, \mathbf{r}}$ be an order $3$ tensor train variety with $\mathbf{d} =(d_1, d_2, d_3)$ and $(r_1,r_2) = (r,1)$ or $(r_1, r_2) = (1, r)$.  
 Then
 $$ \deg(\mathrm{TT}_{\mathbf{d}, \mathbf{r}}) = \binom{(d_1+d_2-r)r+d_3-2}{d_3-1}\prod_{i=1}^r \frac{(d_1+d_2-2r+i-1)! (i-1)!}{(d_1-i)! (d_2-i)!}$$
 or respectively
 $$ \deg(\mathrm{TT}_{\mathbf{d}, \mathbf{r}}) = \binom{(d_2+d_3-r)r+d_1-2}{d_1-1}\prod_{i=1}^r \frac{(d_2+d_3-2r+i-1)! (i-1)!}{(d_2-i)! (d_3-i)!}.$$
\end{proposition}
\begin{proof}
We treat the case $(r_1, r_2) = (r,1)$ since a symmetric argument applies to the other case. By \cite[Lemma 5.6]{BFHP25}, $\mathrm{TT}_{\mathbf{d}, \mathbf{r}}$ is equal to the Segre embedding of 
$V(d_1, d_2; r) \times \PP^{d_3-1}$ where $V(d_1,d_2;r) \subseteq \PP^{d_1d_2-1}$ is the determinantal variety of all $d_1 \times d_2$ matrices of rank $\leq r$. It is known that 
$\dim(V(d_1, d_2;r)) = (d_1+d_2-r)r-1$ \cite[Theorem 3.4.6]{BCRV22} and 
$\deg(V(d_1,d_2;r)) = \prod_{i=1}^r \frac{(d_1+d_2-2r+i-1)! (i-1)!}{(d_1-i)! (d_2-i)!}$ \cite[Theorem 4.4.2]{BCRV22}.  In general, 
for projective varieties $W_1 \subset \PP^{n_1-1}$
and $W_2 \subset \PP^{n_2-1}$ with $\dim(W_1) = k$
and $\dim(W_2)=\ell$, the degree of the Segre embedding of $W_1 \times W_2$ in $\PP^{n_1n_2-1}$ 
is equal to $\binom{k+\ell}{\ell} \deg(W_1)\deg(W_2)$. We use this with $\dim(\PP^{d_3-1})=d_3-1$ and 
$\deg(\PP^{d_3-1})=1$ to obtain the desired formula.  
\end{proof}

\begin{table}[h!]
\centering
\scriptsize
\begin{minipage}[t]{0.48\textwidth}
\centering
\begin{tabular}{@{}ccrrr@{}}
\toprule
$\mathbf{d}=(d_1,d_2,d_3)$ & $\mathbf{r}=(r_1,r_2)$ & $\dim$ & degree & time (s) \\
\midrule
$(3,3,3)$ & $(2,2)$ & 16 & 306       & 0.024 \\
$(3,4,3)$ & $(2,2)$ & 20 & 990       & 0.010 \\
$(3,3,4)$ & $(2,3)$ & 23 & 984       & 0.018 \\
$(3,6,3)$ & $(2,2)$ & 28 & 5{,}103   & 0.057 \\
$(4,2,4)$ & $(2,2)$ & 16 & 1{,}830   & 0.203 \\
$(4,2,4)$ & $(2,3)$ & 19 & 1{,}008   & 0.043 \\
$(4,2,4)$ & $(3,3)$ & 24 & 652       & 0.023 \\
$(4,4,3)$ & $(2,2)$ & 22 & 14{,}640  & 0.432 \\
$(4,4,3)$ & $(3,2)$ & 29 & 4{,}680   & 0.101 \\
\bottomrule
\end{tabular}
\end{minipage}
\hfill
\begin{minipage}[t]{0.48\textwidth}
\centering
\begin{tabular}{@{}ccrrr@{}}
\toprule
$\mathbf{d}=(d_1,d_2,d_3)$ & $\mathbf{r}=(r_1,r_2)$ & $\dim$ & degree & time (s) \\
\midrule
$(4,3,4)$ & $(2,2)$ & 20 & 32{,}316  & 5.579 \\
$(4,3,4)$ & $(2,3)$ & 25 & 19{,}584  & 0.687 \\
$(4,3,4)$ & $(3,3)$ & 33 & 9{,}712   & 0.186 \\
$(3,4,5)$ & $(2,2)$ & 24 & 132{,}660 & 8.697 \\
$(3,4,5)$ & $(2,3)$ & 32 & 121{,}380 & 10.019 \\
$(3,4,5)$ & $(2,4)$ & 38 & 10{,}650  & 2.664 \\
$(4,4,4)$ & $(2,2)$ & 24 & 261{,}560 & 61.593 \\
$(4,4,4)$ & $(2,3)$ & 31 & 151{,}440 & 6.200 \\
$(4,4,4)$ & $(3,3)$ & 42 & 61{,}520  & 1.368 \\
\bottomrule
\end{tabular}
\end{minipage}
\caption{Degrees of order 3 tensor train varieties $\mathrm{TT}_{\mathbf{d},\mathbf{r}}$, computed via the non-intersecting path count of Proposition~\ref{prop:degree-of-order-3}, with running times.}
\label{tab:tt-degrees}
\end{table}

\bibliographystyle{alpha}
\bibliography{bibliography}

\noindent{\bf Authors' addresses:}
\smallskip
\small

\noindent Serkan Hoşten,
San Francisco State University
\hfill {\tt serkan@sfsu.edu}

\noindent Niharika Chakrabarty Paul, Max Planck Institute, Leipzig 
\hfill {\tt niharika.paul@mis.mpg.de}

\noindent Otto T.~P.~Schmidt, 
Max Planck Institute, Leipzig
\hfill {\tt otto.schmidt@mis.mpg.de}

\noindent Dmitry Skurt,
San Francisco State University
\hfill {\tt dskurt@sfsu.edu}
\end{document}